\documentclass[a4paper,12pt,reqno]{article}
 \usepackage[english]{babel}
  \usepackage{amsmath,amsfonts,amssymb,amsthm,bbm}
   \usepackage{graphics,epsfig,psfrag} 
   \usepackage{paralist,subfigure}
   \usepackage[dvipsnames]{xcolor}
   \usepackage{hyperref} 

    \usepackage[active]{srcltx} 
    \usepackage{color,soul} 
    \usepackage[utf8]{inputenc}
\usepackage[T1]{fontenc}
    \usepackage{authblk}
    \usepackage{esint}
    \usepackage{comment}

\newtheorem{theorem}{Theorem}[section]
\newtheorem{corollary}[theorem]{Corollary}
\newtheorem{lemma}[theorem]{Lemma}
\newtheorem{prop}[theorem]{Proposition}

\theoremstyle{definition}

\newtheorem{remark}[theorem]{Remark}

\def\S{\mathbb{S}^{N-1}}
\def\N{\mathbb{N}}
\def\Z{\mathbb{Z}}

\def\R{\mathbb{R}}

\let\e=\varepsilon

\let\t=\tilde
\let\ol=\overline
\let\ul=\underline
\let\.=\cdot
\let\0=\emptyset

\let\mc=\mathcal

\def\1{\mathbbm{1}}

\def\cv{\underset{t\to+\infty}{\longrightarrow}}

\newenvironment{formula}[1]{\begin{equation}\label{#1}}
                       {\end{equation}\noindent}

\def\Fi#1{\begin{formula}{#1}}
\def\Ff{\end{formula}\noindent}

\title{\bf Long-time behavior of the heterogeneous SIRS epidemiological model}
\author[1]{Romain {\sc Ducasse}
}
\author[1]{Maxime {\sc Laborde}}
\affil[1]{Université Paris Cité, CNRS, Sorbonne Université, Laboratoire Jacques-Louis Lions (LJLL), F-
75006 Paris, France}

\date{}

\begin{document}

\maketitle


\begin{abstract}
   We study the long-time behavior of solutions of the SIRS model, a reaction-diffusion system that appears in epidemiology to describe the spread of epidemics. We allow the system to be heterogeneous periodic. Under some hypotheses on the coefficients, we prove that the solutions converge to an equilibrium that we identify and establish some estimates on the speed of propagation.
\end{abstract}


\section{Introduction}

\subsection{Presentation of the problem}

This paper is dedicated to the study of the following heterogeneous parabolic system
\begin{equation}\label{PDE1}
\left\{
\begin{array}{rll}
\partial_t S(t,x) &= d\Delta S(t,x) - \alpha(x) S(t,x) I(t,x) + \lambda(x) R(t,x),\quad &t>0,\ x\in \R^N,\\
\partial_t I(t,x) &= d \Delta I(t,x) + \alpha(x) S(t,x) I(t,x) - \mu(x) I(t,x), \quad &t>0,\ x\in \R^N,\\
\partial_t R(t,x) &= d \Delta R(t,x) + \mu(x) I(t,x) -\lambda(x) R(t,x), \quad &t>0,\ x\in \R^N,
\end{array}
\right.
\end{equation}
where $d>0$ and where the parameters $\alpha,\lambda,\mu$ are periodic functions of the space variable $x$.\\

This system appears in mathematical epidemiology, where it is known as the SIRS system. In this setting, the quantities $S(t,x),I(t,x),R(t,x)$ represent the number of individuals in a given population who are respectively {\em Susceptible} (they do not have the disease but can be infected), {\em Infectious/Infected} (they have the disease and can transmit it) and {\em Recovered} (they had the disease but are now healed) at time $t$ and at position $x$. 

When a susceptible encounters an infectious individual, the susceptible individual can be contaminated with some rate $\alpha>0$ and turned into an infectious individual, hence the {\em mass-action} term $-\alpha SI$ in the equation for $S$ and $+\alpha S I$ in the equation for $I$. The infectious individuals cease to be infectious with some rate $\mu>0$ (then, $\frac{1}{\mu}$ can be seen as the average duration of the infection). Individuals who cease to be infectious become recovered. The recovered individual can not be directly contaminated, they are immune for some time. They lose their immunity with some rate $\lambda>0$, and become susceptible again (they can be infected again).\\

The SIRS system is a {\em compartmental system}. Such systems were originally introduced by Kermack and McKendrick in the papers \cite{KmcK1, KmcK2, KmcK3}. These models are now a cornerstone of mathematical epidemiology. For more details on the epidemiological framework, we refer to the books \cite{Murray1, Murray2, Perthame}. However, these original models were ODE systems, they did not take space into account. In many applications, the effects of space and spatial heterogeneities can not be neglected.\\

In this paper we investigate the SIRS model with a spatial structure, possibly heterogeneous periodic. This allows to consider situations where the key features of the model (the rates of recovery and of contamination, the initial density of population) vary from places to places. These variations can result from different public policies in different places, or they can reflect the influence of the geography, etc. We refer to \cite{Murray2} for more details on the importance of heterogeneities in models in mathematical biology.\\

The model \eqref{PDE1} is a reaction-diffusion system. Such systems appear in the modelling of a variety of phenomenon in chemistry, physics, biology, etc.
Unlike scalar reaction-diffusion equations, for which many general theories have been developed (see Section \ref{related results}) to describe the long-time behavior of solutions, most questions are still open for systems. This is because systems usually lack of comparison principles, upon which most of the theory for reaction-diffusion equations is built.\\

The main result of this paper concerns the long-time behavior of \eqref{PDE1}. We prove that, {\em provided the parameter $\lambda$ is large enough} (we shall give estimates on the required bounds), then the behavior of the system \eqref{PDE1} is characterized by the value of the principal eigenvalue of some operator. If this eigenvalue is strictly negative, the epidemic spreads: there exists a strictly positive stationary periodic state $(S^\star,I^\star,R^\star)$, and the solutions of \eqref{PDE1} converge toward this stationary state. We also give estimates on the {\em spreading speed} of the solutions (we define this notion in Section \ref{related results}). On the other hand, if the mentioned eigenvalue is non-negative, then the solutions converge to a disease-free equilibrium: the epidemic does not spread. In some sense, our result generalizes the classical {\em Freidlin-Gartner} result (see Theorem \ref{th FG} in Section \ref{related results}) to the case of our reaction-diffusion system, under the hypothesis that $\lambda$ is large enough. \\

Let us mention that, even in the case where the coefficients are constants (that is, when $\alpha,\mu,\lambda \in \R$), as far as we are aware, the question of convergence of the solutions was not known.

\subsection{Related results}\label{related results}

We gather here some results on reaction-diffusion equations and systems, in order to explain how our results fit in this context.

\subsubsection{Reaction-diffusion equations}

Reaction-diffusion equations are semilinear parabolic equations of the form
\begin{equation}\label{eq rd}
    \partial_t u = d\Delta u +f(u),\quad t>0,\ x\in \R^N.
\end{equation}
When the function $f$ is the concave function $f(u)=u(1-u)$, equation \eqref{eq rd} is known as the Fisher-Kolmogorov-Petrovski-Piskunov equation. These authors studied this equation in their seminal papers \cite{KPP, Fisher} and proved that, for any initial datum $u_0$ which is continuous, non-negative, non-zero and compactly supported (a \emph{small initial disturbance} of the state $u\equiv 0$), propagation occurs in the sense that
$$
\sup_{\vert x \vert < ct}\vert u(t,x) - 1\vert \underset{t\to+\infty}{\longrightarrow} 0, \quad \forall c\in [0,2\sqrt{d}),
$$
and
$$
\sup_{\vert x \vert > ct}\vert u(t,x) \vert \underset{t\to+\infty}{\longrightarrow} 0, \quad \forall c> 2\sqrt{d}.
$$
This is a {\em spreading} result. This means that the solution $u$ converges to the stationary solution $1$ and does so with speed $c^\star = 2\sqrt{d}$. Indeed, the level sets of the solution $\{x \in \R^N \ : \  u(t,x) = z\}$ ($z \in (0,1)$) grow like balls with radius $ct$, that is, they expand in all directions with speed $c$ (the propagation is isotropic).

When the nonlinearity $f$ is a different function, the situation is more involved, see \cite{AW} for instance.\\

The case of heterogeneous equations was also considered, and many results were obtained, we refer to \cite{BHN, BHNadirI, BHNadirII} and the references therein. A very salient result is the {\em Freidlin-Gartner formula} \cite{FG}, which generalises the spreading result mentioned above. We state it here and we will use it several times in the sequel. 
\begin{theorem}[Freidlin-Gartner]\label{th FG}
    Let $\gamma, \alpha \in C^\delta_{per}$\footnote{Here and in the sequel, for $\delta \in (0,1)$, $C^\delta_{per}$ denotes the set of functions which are $\delta$-Hölder continuous and $1$-periodic with respect to the $x$ variable, that is, $f(x+k)=f(x)$ for all $k\in \Z^N$. We only consider $1$-periodicity for simplicity, doing otherwise would not change the analysis.}, for some $\delta>0$ and assume that $\alpha>0$. Let $u(t,x)$ be the solution of 
    \begin{equation}\label{eq KPP per}
    \partial_t u = d\Delta u + \gamma u - \alpha u^2, \quad t>0, \ x\in \R^N,
    \end{equation}
    with initial datum $u(0,\cdot)$ continuous, compactly supported, non-negative and non-zero and let $\lambda_1$ be the principal periodic eigenvalue of the operator $-d\Delta - \gamma$. 
    
    Assume that, 
    $$\lambda_1<0,$$
    then:
    
    \begin{itemize}
        \item There is a unique positive stationary solution to \eqref{eq KPP per}, name it $p(x)$. This function is $1$-periodic with respect to $x$.

    \item We have 
    $$u(t,x)\underset{t\to+\infty}{\longrightarrow} p(x)$$ as $t$ goes to $+\infty$, locally uniformly in $x$.

   \item Moreover, there is $w^\star$, a positive continuous function from the sphere $\S$ to $\R^+_\star$ such that the propagation occurs with spreading speed $w^\star(e)$ in the direction $e\in \S$, in the sense that
    $$
    \sup_{\substack{x = r e \\ r \in [0,(w^\star(e)-\e)t], \, e \in \S}} \vert u(t,x) - p(x)\vert \underset{t\to+\infty}{\longrightarrow} 0,\quad \forall \e>0,
    $$
    and
    $$
    \sup_{\substack{x = r e \\ r \geq (w^\star(e)+\e)t, \, e \in \S}} \vert u(t,x) \vert \underset{t\to+\infty}{\longrightarrow} 0,\quad \forall \e>0.
    $$

    \item In addition, the function $w^\star$ above is given by the Freidlin-Gartner formula: 
\begin{equation}\label{eq FG}
    w^\star(e) := \inf_{\rho \in \R^N} \frac{-k(\rho)}{\rho\cdot e},
\end{equation}
where $k(\rho)$ is the principal eigenvalue of the following operator acting on $C^\delta_{per}$,
\begin{equation}
\label{eq operator L_lambda} 
    L_\rho(\psi) = - d\Delta \psi + 2d \rho \cdot \nabla \psi -(d \vert \rho \vert^2 + \gamma(x))\psi.
\end{equation}

\end{itemize}

On the other hand, if $\lambda_1 \geq 0$, then there are no positive stationary solutions to \eqref{eq KPP per}, and the solutions of \eqref{eq KPP per} converge to $0$ uniformly in $x$ as $t$ goes to $+\infty$.
    
\end{theorem}
%
%
Theorem \ref{th FG} was originally obtained by Freidlin and Gartner \cite{FG} by mean of probabilistic tools. It was later proved by Berestycki, Hamel and Nadin in \cite{BHN} using a PDE approach. Similar properties of spreading for heterogeneous reaction-diffusion equations have been studied with other approaches: the viscosity solution/singular perturbation method is adopted by Evans and Souganidis in \cite{ES} and Barles, Soner and Souganidis in \cite{BSS}. Weinberger uses an abstract discrete system approach in \cite{W}. The result was extended to more general nonlinearities by Rossi \cite{Rossi} and to more general domains by the first author in \cite{Duc1}.\\

In some sense, our main result (Theorem \ref{main th} below) partially extends the Freidlin-Gartner theorem to the system \eqref{PDE1}, under the hypothesis that $\lambda$ is large enough. We emphasize that we expect the Freidlin-Gartner result to be false for systems of reaction-diffusion in general.\\

\paragraph{Periodic principal eigenvalues.}
Let us give a quick reminder on periodic principal eigenvalues, as they play an important role in the sequel. For $\gamma \in C^\delta_{per}$, the principal periodic eigenvalue of the elliptic operator
$$
\phi \mapsto -d\Delta \phi - \gamma \phi
$$
designs its smallest eigenvalue. The existence of this eigenvalue comes from the Krein-Rutman theorem applied to elliptic equations, see \cite{GT}: it states that such elliptic operators admit a unique eigenvalue associated with a positive eigenfunction, that this eigenvalue is simple and that it is the smallest eigenvalue.

For a given operator $L$, we denote $\lambda_1(L)$ its principal eigenvalue. When the operator is a symmetric operator of the form $-\Delta - \gamma$, the principal periodic eigenvalue is given by the classical Rayleigh-Ritz quotient ($ C^1_{per}$ is the set of $C^1$ functions that are periodic)
$$
\lambda_1(-d\Delta -\gamma) = \inf_{\phi \in C^1_{per}} \frac{\int_{[0,1]^N} d\vert \nabla \phi \vert^2 - \gamma \phi^2}{\int_{[0,1]^N} \phi^2}.
$$
We mention that there exist other variational formulae that are valid also when the operator is not symmetric (for instance if there is a drift term). See \cite{BR} for more results on the use of principal eigenvalues in reaction-diffusion theory.\\

Observe that, in the homogeneous setting, that is when $\gamma(x) \equiv \ol \gamma \in \R$ and $\alpha(x) \equiv \ol \alpha \in \R$, then the principal periodic eigenvalue $\lambda_1$ of $-d\Delta -\ol \gamma$ is simply $-\ol \gamma$, therefore the condition to have propagation in \eqref{eq KPP per} is to have $\ol \gamma >0$. Also, the principal periodic eigenvalue of $L_\lambda$ is $-(d\vert \lambda \vert^2 + \ol \gamma)$, and it is easy to check that in this case
$$
w^\star(e) = \inf_{r>0} \frac{d r^2 + \ol \gamma}{r} = 2\sqrt{d\ol \gamma},
$$
that is, the speed is isotropic, we find the result of Kolmogorov, Petrovski and Piskunov mentioned above (with $\ol\gamma=1)$.

\subsubsection{Reaction-diffusion systems and epidemiology}

When it comes to systems of reaction-diffusion equations, the situation is much more involved. Indeed, a main tool in proving Theorem \ref{th FG} is the parabolic comparison principle. Such principle does not hold for systems in general. However, using a variety of techniques, many results were obtained for compartmental models. \\

The compartmental model which was the most studied is probably the SIR model, which is system \eqref{PDE1} with $\lambda \equiv 0$ (there is no waning of immunity). It reads as
\begin{equation}\label{SIR}
\left\{
\begin{array}{rll}
\partial_t S &= d_1\Delta S - \alpha S I ,\quad &t>0,\ x\in \R^N,\\
\partial_t I &= d_2 \Delta I + \alpha S I - \mu I, \quad &t>0,\ x\in \R^N.
\end{array}
\right.
\end{equation}
One can also write an equation for the $R$ population, but because these individuals are immune, they do not really play any role in the system.\\

Although simple looking, the model is not easy to study. For instance, proving that the solutions are bounded uniformly in $(t,x)\in \R^+\times \R^N$ - which would be the first step to have any hope to study the long-time behavior of solutions - is not so easy when $d_1\neq d_2$ (if $d_1=d_2$ one can sum the equations and apply the parabolic comparison principle to the resulting equation). We refer to \cite{Herrero} for details on this topic.\\

However, the existence of traveling waves for \eqref{SIR} was proven by Hosono and Ilyas \cite{HI} (the study boils down to a system of ODE which can be studied using adequate methods). Later, Källen \cite{Kallen} proved in some cases propagation results similar to the one considered here. These results were concerned with the homogeneous case. 

Some results were established in the heterogeneous setting, we refer to \cite{Ducrot, DG}.\\ 

Let us also mention the work \cite{BRRSIR} where the authors consider the influence of networks on similar  models. We also mention that models where contamination occurs at distance were first suggested by Kendall \cite{K1, K2}, and were studied in the homogeneous case by Diekmann and Thieme, see \cite{Diek}, \cite{T} and in the heterogeneous case by the first author, see \cite{DucHetSir}.\\

Our system, the SIRS model \eqref{PDE1} was less studied. For instance, we are not aware of the proof of existence of traveling waves for \eqref{PDE1} (which is usually the first topic studied, as it boils down to the study of ODEs). The difference with the SIR system \eqref{SIR} is that the individuals can be contaminated several times, leading to some oscillations in the dynamics, see Figure \ref{fig}.\\

\begin{figure}[h!]

\begin{tabular}{@{\hspace{0mm}}c@{\hspace{1mm}}c@{\hspace{1mm}}c@{\hspace{1mm}}}

\centering
\includegraphics[ scale=0.30]{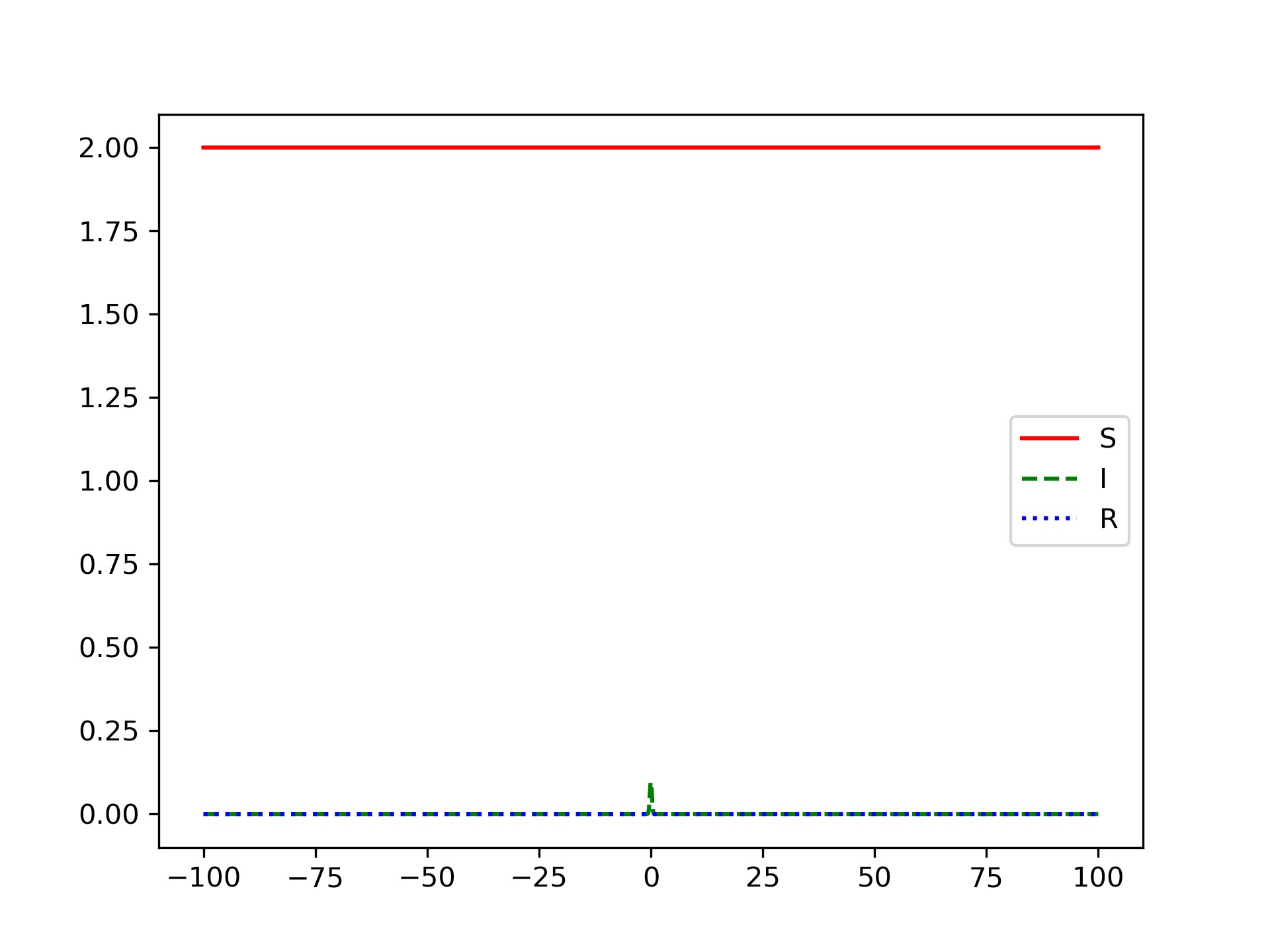}&
\includegraphics[ scale=0.30]{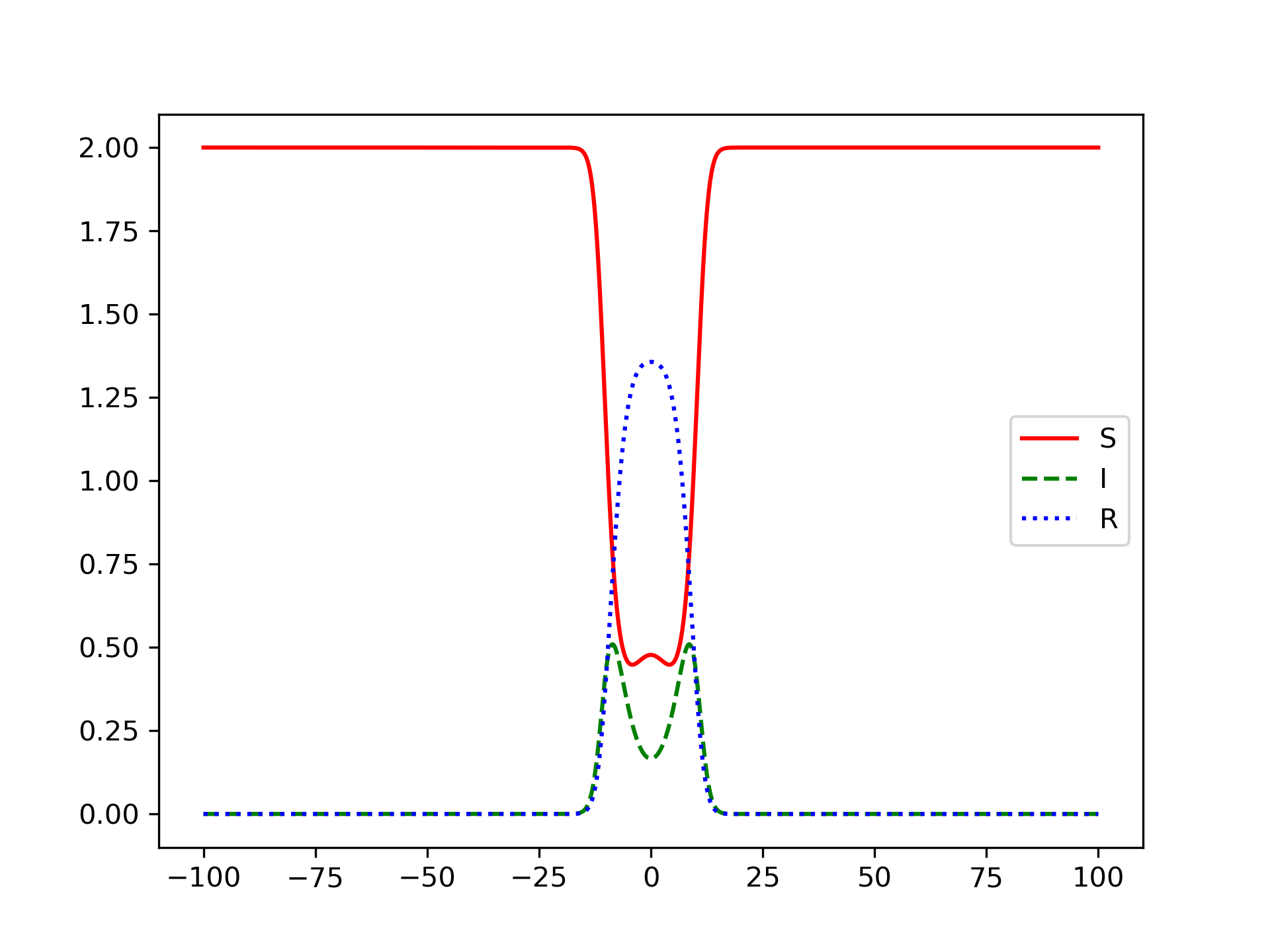}&
\includegraphics[ scale=0.30]{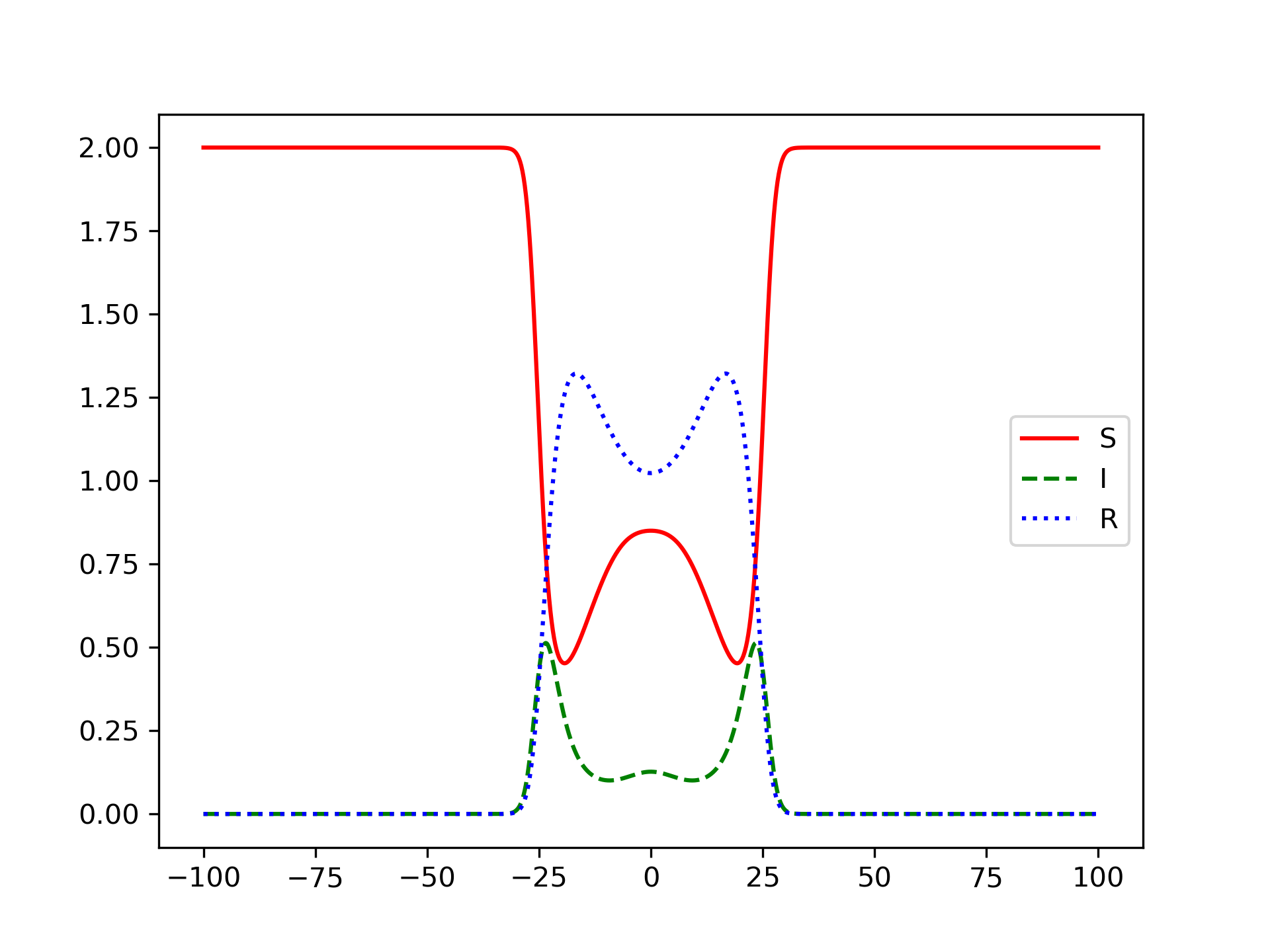}\\
\includegraphics[ scale=0.30]{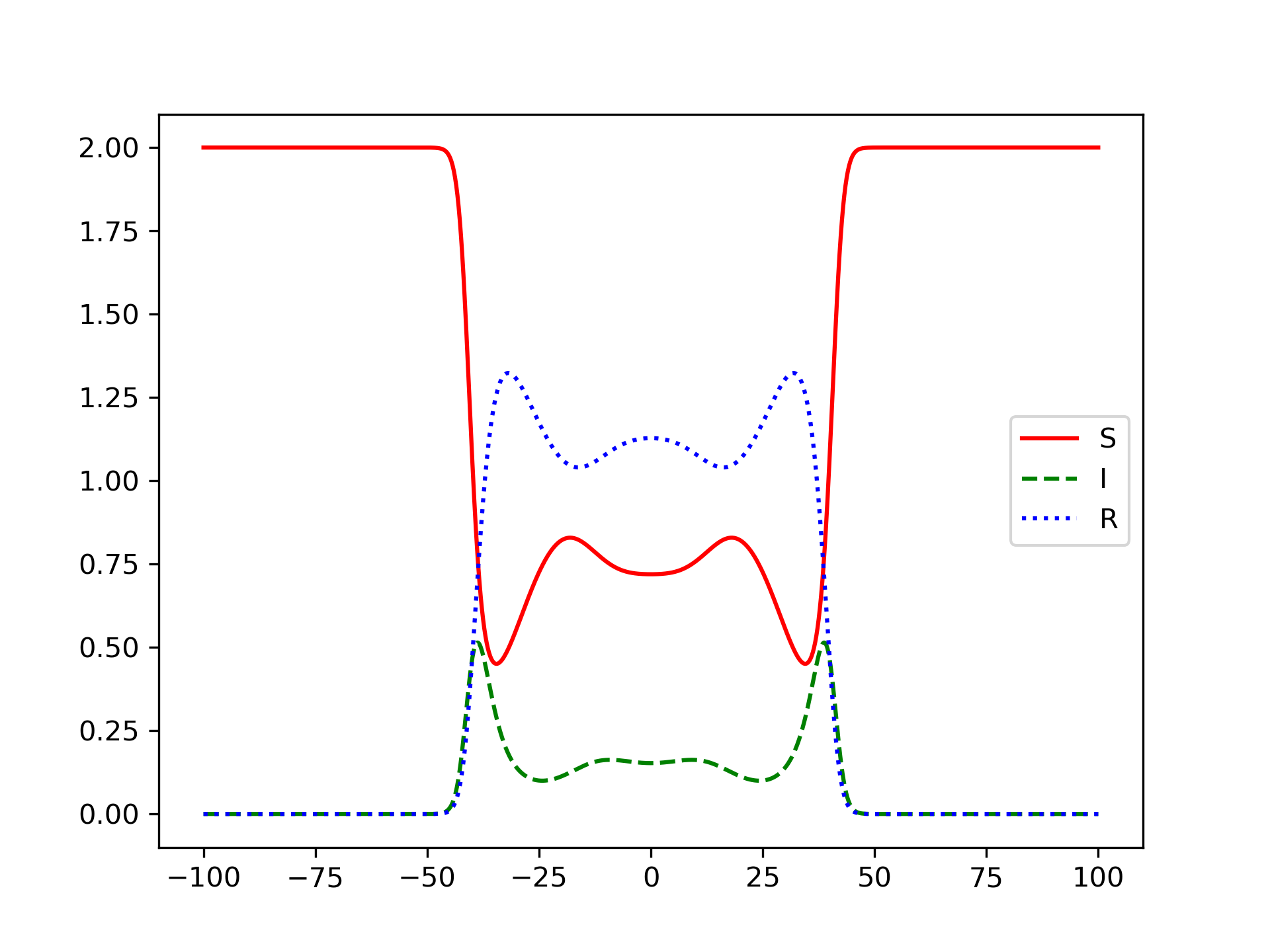}&
\includegraphics[ scale=0.30]{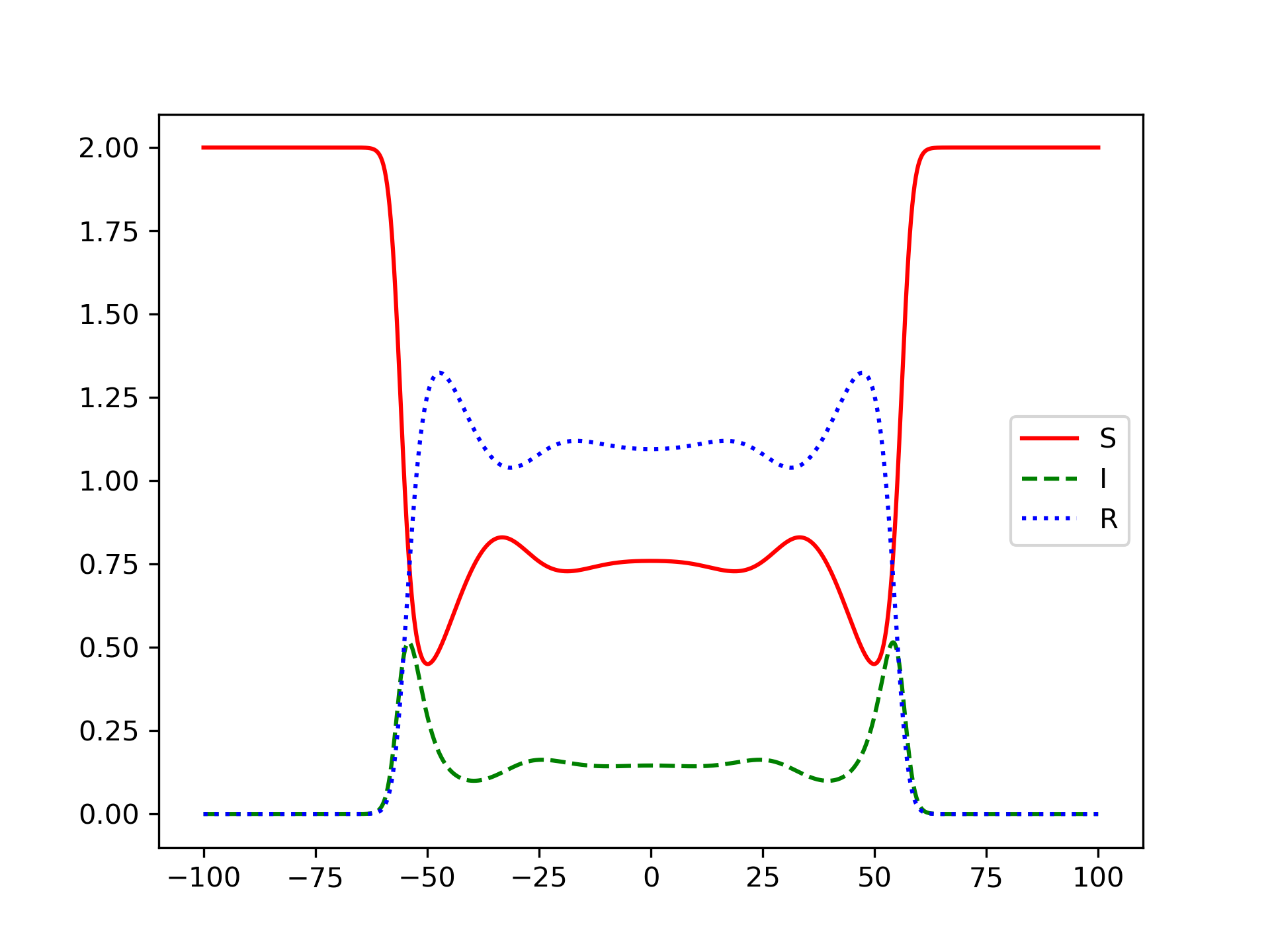}&
\includegraphics[ scale=0.30]{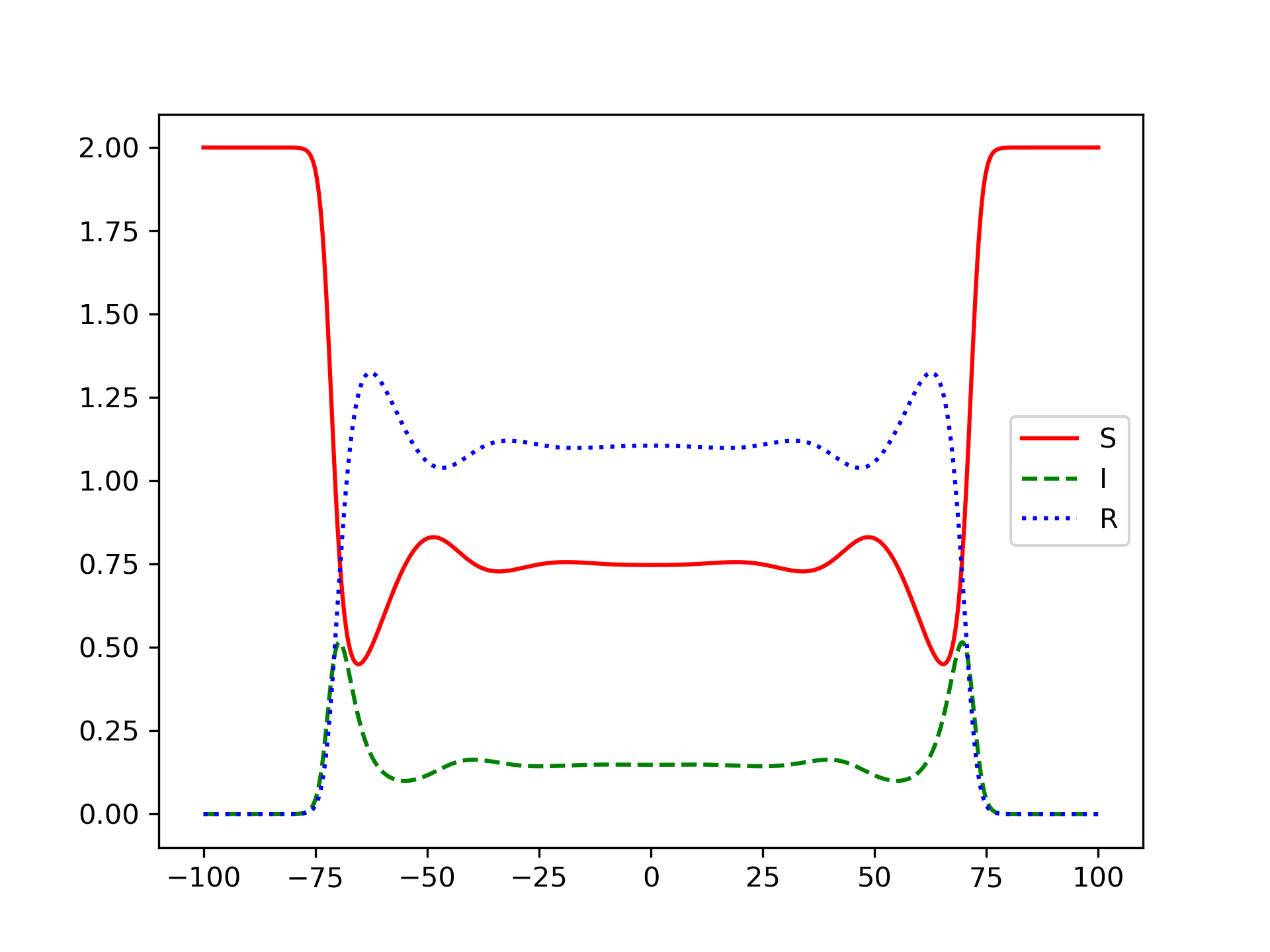}\\
\end{tabular}
\caption{\textit{Evolution of the SIRS model. Starting from a small density of infectious and a constant density of susceptible, the disease spreads by forming front-like solutions.}}
\label{fig}
\end{figure}

We now present the main result of this paper, which can be seen as an extension of Theorem \ref{th FG} to the case of SIRS models.

\subsection{Main result}

In the whole paper, we assume that the functions $\alpha,\mu,\lambda$ in \eqref{PDE1} belong to $C^\delta_{per}$, for some $\delta>0$ fixed and are all strictly positive. The diffusion constant $d$ is also supposed strictly positive.

When working with the evolution system \eqref{PDE1}, we need to complete it with an initial datum for $S,I,R$ at time $t=0$. We shall always consider 
$$ 
S(0,\cdot)=S_0, \quad I(0,\cdot)= I_0, \quad R(0,\cdot)=0,
$$ 
with $S_0 \in C^\delta_{per}$, $S_0>0$ and $I_0$ non-negative, non-zero, continuous and compactly supported. It is natural to consider such initial data: indeed, before the beginning of the epidemic, there are only few infected individuals and they are localized. Moreover, there should be no recovered individuals yet.\\

 We shall take advantage of the fact that the diffusion constants in the three equations in \eqref{PDE1} are equal to observe that
$$
N(t,x) := S(t,x)+I(t,x)+R(t,x)
$$
solves the heat equation on $\R^N$, that is, $\partial_t N = d\Delta N$, and therefore $N(t,x)$ converges uniformly toward $\fint S_0$ as $t\to +\infty$ (we denote $\fint = \lim_{R\to+\infty} \frac{1}{\vert B_R\vert}\int_{B_R}$ the average on the whole space). This allows to rewrite the system \eqref{PDE1} as the following

\begin{equation}\label{PDE2}
\left\{
\begin{array}{rll}
\partial_t I &= d \Delta I + \gamma(t,x) I - \alpha I^2 - \alpha I  R, \quad &t>0,\ x\in \R^N,\\
\partial_t R &= d \Delta R + \mu I -\lambda R, \quad &t>0,\ x\in \R^N,
\end{array}
\right.
\end{equation}
with $ \gamma(t,x) :=\alpha(x)N(t,x) - \mu(x)$. Clearly, we have 
\begin{equation}\label{def gamma}
 \gamma(t,x) \underset{\substack{t\to +\infty \\ \text{Unif}}}{\longrightarrow} \gamma^\star(x) := \alpha(x)\left(\fint S_0\right) - \mu(x).
\end{equation}
In order to state our main result, we need to introduce some notations. First, we denote $\ol I$ the largest bounded solution to
\begin{equation}\label{eq ol I}
    -d\Delta \ol I = \gamma^\star \ol I - \alpha \ol I^2,
\end{equation}
and we denote $\ol R$ the solution of
\begin{equation}\label{eq ol R}
    -d\Delta \ol R +\lambda \ol R= \mu \ol I.
\end{equation}
It follows from Theorem \ref{th FG} that $\ol I,\ol R$ are strictly positive and periodic if and only if
\begin{equation}\label{assumption1}
    \lambda_1(-d\Delta -\gamma^\star )<0,
\end{equation}
they are everywhere equal to $0$ otherwise. We shall also sometimes require the following to hold true
\begin{equation}\label{assumption2}
  \lambda_1(-d\Delta -(\gamma^\star - \alpha \ol R))<0.
\end{equation}
It turns out that \eqref{assumption2} implies \eqref{assumption1} because $\alpha \ol R \geq 0$ (this can be readily inferred from the expression of the principal eigenvalue given by the Rayleigh-Ritz formula).

When \eqref{assumption1} and \eqref{assumption2} hold true, we can define the two following speed of propagation $w_\star, w^\star$
\begin{equation}\label{def w haut}
w^\star \ \text{is the speed of propagation of }\ \partial_t u = d\Delta u +\gamma^\star u - \alpha u^2,
\end{equation}
and
\begin{equation}\label{def w bas}
  w_\star \ \text{is the speed of propagation of }\   \partial_t u = d \Delta u +(\gamma^\star - \alpha  \ol R)u - \alpha u^2.
\end{equation}
Owing to Theorem \ref{th FG}, $w_\star, w^\star$ can be expressed using the formula \eqref{eq FG}. Observe that we have $w_\star \leq w^\star$ (this can either be seen from the formula \eqref{eq FG}, or it can be directly seen as a consequence of the parabolic comparison principle: any solution of the equation in \eqref{def w haut} is supersolution for the equation in \eqref{def w bas}, and then spreads faster).

We finally define the quantity $\Lambda_0\in \R$ by
\begin{equation}\label{est lambda}
    \Lambda_0 = \frac{\max\{\mu \}\max\{ \frac{\gamma^\star}{\alpha}\}}{\min\{ \frac{\gamma^\star}{ \alpha}\}}
    \left(\frac{\max\{\alpha\}}{\min\{\alpha\} } + 1\right).
     \end{equation}
%

We are now in position to state our main result.

\begin{theorem}\label{main th}
Let $S_0, I_0$ be fixed, where $S_0 \in C^\delta_{per}$ is strictly positive and $I_0$ is continuous, compactly supported, non-negative, non-zero. Let $\gamma^\star$ be defined by \eqref{def gamma} and $(S,I,R)$ be the solution of the system \eqref{PDE1} with initial datum $(S_0,I_0,0)$. 
\begin{itemize}
    \item Assume that \eqref{assumption1} holds true and that $\min_{x\in\R^N}\lambda(x) > \Lambda_0$. Then, there is $(S^\star,I^\star,R^\star)$ stationary solution of \eqref{PDE1}, which is strictly positive and periodic. In addition,
$$
(S(t,x),I(t,x),R(t,x)) \underset{\substack{t\to+\infty \\ \text{loc. unif.}}}{\longrightarrow} (S^\star(x),I^\star(x),R^\star(x)).
$$
Moreover, denoting $w^\star, w_\star$ the speeds defined by \eqref{def w haut}, \eqref{def w bas}, we have, for all $\e >0$,
$$
\sup_{\substack{x = r e \\ r\in [0,(w_\star(e)-\e)t] , \ e\in \S}} \vert S - S^\star\vert + \vert I - I^\star \vert + \vert R - R^\star \vert \underset{t\to+\infty}{\longrightarrow} 0,
$$
and
$$
\sup_{\substack{x = r e \\ r\geq (w^\star(e)+\e)t , \ e\in \S}} \left\vert S - \fint S_0\right\vert + \vert I \vert + \vert R  \vert \underset{t\to+\infty}{\longrightarrow} 0.
$$

\item If \eqref{assumption1} is not verified, i.e, when $\lambda_1(-d\Delta -\gamma^\star)\geq 0$, then 
$$
(S(t,x),I(t,x),R(t,x)) \underset{\substack{t\to+\infty}}{\longrightarrow} \left(\fint S_0,0,0\right),
$$
uniformly in $x \in \R^N$.
\end{itemize}

\end{theorem}
%
Let us state some comments on this result. Consider the situation when \eqref{assumption1} is not verified, that is, $\lambda_1(-d\Delta -\gamma^\star)\geq 0$.  In this case, the density $I$ goes to zero uniformly, this means that no propagation of the epidemic occurs (the only infections occurring are due to the initial density of infectious individuals). The susceptible individuals are therefore almost only subject to diffusion, and $S(t,x)$ converges to its average, $\fint S_0$. 
%

In the other case, when the epidemic spreads, the density of infectious invades space with a speed comprised between $w_\star$ and $w^\star$. The density of recovered follows the density of infectious. The dynamics of the susceptible is slightly different: at the initial time, there are already susceptible individuals everywhere in space. In the region $\{x = re, \ r\in [0,(w_\star(e)-\e)t],\ e\in \S\}$, the infectious spread and therefore consume a portion of the susceptible, who stabilize toward the limit state $ S^\star>0$. However, in the region $\{x = re, \ r>(w^\star(e)+\e)t,\ e\in \S\}$, the infectious are not there yet, so no contamination occurs. However, the susceptible are subject to diffusion everywhere in space, so they homogenize, independently of the epidemic. Therefore, in this region untouched by the epidemic (which is receding), the susceptible density asymptotically looks like $\fint S_0$.\\

As far as we are aware, even in the homogeneous case, our result Theorem \ref{main th} is new. In this case, it can be restated under a simpler and more explicit form (and we have a better lower bound for $\Lambda_0)$.

\begin{corollary}\label{cor hom}
    Assume that $\alpha,\mu,\lambda,S_0$ are positive constants. Let $I_0$ be continuous, compactly supported, non-negative, non-zero. Let $(S,I,R)$ be the solution of \eqref{PDE1} with initial datum $(S_0,I_0,0)$. 

    \begin{itemize}
        \item If $\alpha S_0 - \mu >0$ and $\lambda > \mu$, then define
    $$
    c^\star=2\sqrt{d(\alpha S_0 - \mu)}\quad \text{and}\quad c_\star = 2\sqrt{d(\alpha S_0 - \mu)\left(1-\frac{\mu}{\lambda}\right)},
    $$
    and
    $$
    (S^\star,I^\star,R^\star) = \left(\frac{\mu}{\alpha} , \frac{\alpha S_0 - \mu}{\alpha(1+\frac{\mu}{\lambda})},\frac{\mu}{\lambda}\frac{\alpha S_0 - \mu}{\alpha(1+\frac{\mu}{\lambda})} \right).
    $$
    We have, for all $\e >0$,
$$
\sup_{\vert x \vert \leq (c_\star - \e)t} \vert S - S^\star\vert + \vert I - I^\star \vert + \vert R - R^\star \vert \underset{t\to+\infty}{\longrightarrow} 0,
$$
and
$$
\sup_{\vert x \vert \geq (c^\star+\e)t} \left\vert S - S_0\right\vert + \vert I \vert + \vert R  \vert \underset{t\to+\infty}{\longrightarrow} 0.
$$

\item If $\alpha S_0 - \mu \leq 0$, then 
$$
(S(t,x),I(t,x),R(t,x)) \underset{\substack{t\to+\infty}}{\longrightarrow} \left( S_0,0,0\right),
$$
uniformly in $x \in \R^N$.
    \end{itemize}
\end{corollary}

The strategy of the proof is the following. First, we establish in Section \ref{sec stat} the existence of a stationary solution. This will be done by rewriting the system as a fixed point problem - the proof will follow classical lines there. The crux of the proof of Theorem \ref{main th} lies in the study of the evolution problem (convergence of solutions and estimates on the spreading speeds). It is done in Section \ref{sec evol} and relies on a bootstrap argument. We first use the parabolic comparison principle on one of the equation to get a crude estimates on one of the solution, then plug this new solution in the other equation, use again the parabolic principle, and so on. The main difficulty is to keep track of the estimates on the spreading speeds and on the convergence of functions at each iteration. The proof of the second point (the non-propagation of the epidemic when \eqref{assumption1} is not verified) is done in Proposition \ref{prop non prop}.

\begin{remark}
The following open questions arise naturally.
\begin{itemize}
    \item First, it  is natural to wonder what happens if the diffusivities on the three equations of \eqref{PDE1} are different. In this case, the total density $N =S+I+R$ is not a solution to the heat equation. It is not even clear whether it is bounded. We believe that Theorem \ref{main th} still holds true, however the proof seems out of reach for now. We mention that the same problem arises for the SIR system (when $\lambda=0$), we refer to \cite{Ducrot} for further details. However, we want to emphasize that the analysis of the system \eqref{PDE2} does not require to have the same diffusivity on the two equation.

    \item We conjecture that the condition $\lambda > \Lambda_0$ is actually not required for Theorem~\ref{main th} to hold true. Up to conducting a more detail analysis, one can indeed weaken this hypothesis, but it is not clear how (if possible) to get rid of it.

    \item We also think that, in the spreading result of Theorem \ref{main th}, one can replace $w_\star$ with $ w^\star$, that is, the spreading should occur at speed exactly $ w^\star$. This is supported by numerical experiments.
\end{itemize}

\end{remark}

\section{Existence of stationary solutions}\label{sec stat}

The goal of this section is to prove that there exist, under conditions, positive stationary solutions to \eqref{PDE1}. In order to do so, we consider the following auxiliary system
\begin{equation}\label{stat}
\left\{
\begin{array}{rll}
 -d\Delta I - \gamma^\star I + \alpha I^2 + \alpha IR &=0, \quad &x\in \R^N,\\
    -d\Delta R + \lambda R - \mu I &=0,\quad &x\in \R^N,
\end{array}
\right.
\end{equation}
where $\alpha,\mu,\lambda$ satisfy the hypotheses of Theorem \ref{main th} and $\gamma^\star \in C^\delta_{per}$ is a function, non necessarily positive.

When studying the evolution system \eqref{PDE1} in the next section, we will apply the results of the present section with $\gamma^\star$ given by \eqref{def gamma}. However, in this section, we can assume $\gamma^\star$ to be any function in $C^\delta_{per}$, and we do not need to make reference to any initial datum $S_0$ for $S$.


We shall prove the following

\begin{theorem}\label{th existence}
Let $\gamma^\star \in C^\delta_{per}$ and assume that \eqref{assumption1} and \eqref{assumption2} hold true. Then, there exists at least one positive periodic solution $(I^\star,R^\star) \in (C^\delta_{per})^2$ to \eqref{stat}.
\end{theorem}
As explained below in Proposition \ref{prop stat}, this theorem yields existence of stationary solutions to \eqref{PDE1} - up to choosing an appropriate $\gamma^\star$.\\

To prove Theorem \ref{th existence}, we rewrite the system \eqref{stat} as a fixed-point problem. To do so, we introduce three operators $T,A,Z$ acting on $C^\delta_{per}$.

\begin{itemize}

\item The operator $Z$ is defined by $Z(I)= R$, where $R$ is the solution to
$$
-d\Delta R +\lambda R =\mu I.
$$
$Z$ is well defined because $\lambda>0$. This operator is linear, and order preserving, that is, if $I_1\geq I_2$, then $Z(I_1)\geq Z(I_2)$.

\item The operator $A$ is the affine transformation
$$
A(R) = \frac{\gamma^\star}{\alpha} -R.
$$
This operator changes the order of inequalities: if $R_1\geq R_2$, then $A(R_1)\leq A(R_2)$ (we introduce this operator for notational simplicity in the coming proofs).

\item The (nonlinear) operator $T$ is defined so that $T(V)=I$, where $I$ is the largest periodic solution of
$$
-d\Delta I = \alpha VI - \alpha I^2.
$$
The well definiteness of $T$ is a consequence of Theorem \ref{th FG} above. If the principal periodic eigenvalue of $-\Delta - \alpha V$ is strictly negative, $T(V)$ is a positive periodic function, else, it is zero. In addition, $T$ is order preserving: if $V_1 \geq V_2$, then $T(V_1)\geq T(V_2)$.

\end{itemize}

It is easy to check that proving the existence of a solution to \eqref{stat} is equivalent to finding a fixed point for the operator $T\circ A \circ Z$ in $C^\delta_{per}$.\\

\begin{remark}\label{req stat}
Using the previous operators, we can write $\ol I = T(\frac{\gamma^\star}{\alpha})$ and $\ol R = Z(\ol I)$ (we recall that $\ol I, \ol R$ are defined above as the solutions to \eqref{eq ol I}, \eqref{eq ol R}). When $\gamma^\star$ is such that \eqref{assumption1} hold true, then 
$\ol I, \ol R$ are strictly positive.

We shall also denote $\ul I$ to be the largest solution of
\begin{equation}\label{eq ul I}
    -d\Delta \ul I = (\gamma^\star - \alpha \ol R) \ul I - \alpha \ul I^2,
\end{equation}
or, with the above notations, $\ul I = T\circ A \circ Z(\ol I)$. When $\gamma^\star$ is such that \eqref{assumption2} is verified, then $\ul I$ is also strictly positive.
\end{remark}

The proof of Theorem \ref{th existence} relies on the following lemma:
\begin{lemma}\label{lem stable}
Assume that $\gamma^\star$ is such that \eqref{assumption1} holds true. Then, the set
    \begin{equation}\label{def I}
    \mc I = \{ I \in C^\delta_{per} \ : \ \ul I \leq I \leq \ol I\}
    \end{equation}
    is non-empty and is left invariant by the composition of the three operators $T\circ A \circ Z$.
\end{lemma}

\begin{proof}
Because \eqref{assumption1} holds true, then $\ol I = T(\frac{\gamma^\star}{\alpha}) >0$. Hence $\frac{\gamma^\star}{\alpha}> \frac{\gamma^\star}{\alpha}- Z(\ol I) = A\circ Z(\ol I)$ and then, because $T$ is order preserving, we find that $\ul I  < \ol I$. Therefore, $\mc I$ is non-empty.

 Now, let $I\in \mc I$. We have
    $$
\ul I \leq I \leq \ol I.
    $$
    Therefore, because $Z,T$ are order-preserving and $A$ changes the order, we get 
    $$
T\circ A \circ Z(\ul I) \geq T\circ A \circ Z(I) \geq T\circ A \circ Z(\ol I).
    $$
    Then, using $\ul I = T\circ A \circ Z(\ol I)$, we obtain 
    $$
     \ul I \leq T\circ A \circ Z(I). 
    $$
    It only remains to prove that
    $$
    T\circ A \circ Z(\ul I) \leq \ol I.
    $$
    We have $\frac{\gamma^\star}{\alpha}\geq \frac{\gamma^\star}{\alpha}-Z(\ul I) = A\circ Z(\ul I)$. Composing this with $T$, and remembering that $\ol I = T(\frac{\gamma^\star}{\alpha})$ give the result.
\end{proof}
We did not use assumption \eqref{assumption2} in the proof of this lemma. This assumption will be needed in the proof of Theorem \ref{th existence} that comes just now: we shall find a fixed point of $T\circ A \circ Z$ in $\mc I$, and in order to be sure that we do not find the trivial fixed point $I=0$, we want to have $\ul I>0$ (so that $0\not\in \mc I)$, which is indeed guaranteed by \eqref{assumption2}.

\begin{proof}[Proof of Theorem \ref{th existence}.]
Owing to Lemma \ref{lem stable}, the set $\mc I \subset C^\delta_{per}$ is stable by $T\circ A \circ Z$, convex and closed. As we mentioned earlier, any fixed point of the map $T\circ A\circ Z$ will be a stationary solution to \eqref{stat}. Let us show that the map $T\circ A\circ Z$ indeed admits a fixed point in the set $\mc I$.

To apply the Schauder fixed-point theorem (see for instance \cite{GT}), we need $T\circ A\circ Z$ to be continuous and compact. This comes directly from standard regularity theory for elliptic equations, that we present here for completeness.

There is $C>0$ such that for any $f \in \mc I$, we have $\|f\|_{L^\infty}\leq C$ (just take $C=\|\ol I\|_{L^\infty}$). It is easy to see that $\|Z(f)\|_{L^\infty} \leq \left(\frac{\max \mu}{\min \lambda}\right)C$ and then $\|A\circ Z(f)\|_{L^\infty}\leq C$, for a possibly different $C$ (still independent of $f$). Then, if $g= T\circ A \circ Z(f)$, we have
$$
-d \Delta g = \alpha A \circ Z(f) g - \alpha g^2.
$$
Owing to the elliptic maximum principle, we have that $\|g\|_{L^\infty }\leq C$ for some $C>0$. Now, owing to $L^p$ estimates for elliptic equations (see \cite{GT}), we get
$$
\forall f \in \mc I, \quad \|T\circ A \circ Z(f)\|_{W^{2,p}}\leq C.
$$
The constant $C$ is still independent of $f$ (it depends only on $d,\alpha,\gamma^\star,\mu,\lambda,\ol I, p$). Up to taking $p$ large enough, the Morrey inequality tells us that the image of $T\circ A\circ Z$ is relatively compact in $\mc I$. Hence the result.
\end{proof}

We conclude this section with two results. The first one asserts that solutions of \eqref{stat} provide stationary solutions to \eqref{PDE1}. The second is a technical result that will be useful later on, it gives an estimate on the Lipschitz norm of $T\circ A \circ Z$.

\begin{prop}\label{prop stat} Let $M>0$ be such that $\gamma^\star(x):=\alpha(x)M-\mu(x)$ verifies the assumptions of Theorem \ref{th existence}, \eqref{assumption1} and \eqref{assumption2}.

Then, there is a positive periodic stationary solutions $(S^\star,I^\star,R^\star)$ to \eqref{PDE1}. This solution verifies $\fint (S^\star+I^\star+R^\star )  = M$.
\end{prop}
In particular, if an initial datum $(S_0,I_0,R_0)$ is specified for \eqref{PDE1}, and if \eqref{assumption1}, \eqref{assumption2} are verified, we can take $M=\fint S_0$ in Proposition \ref{prop stat} to find a stationary solution to \eqref{PDE1} with $\gamma^\star(x) = \alpha(x)\left(\fint S_0\right) - \mu(x)$. We shall see after that this stationary solution is the that effectively attracts the dynamics of the evolution system.
\begin{proof}
Let $\gamma^\star$ be of the specific form
$$
\gamma^\star(x) = \alpha(x)M - \mu(x),
$$
where $M$ is such that the hypotheses of Theorem \ref{th existence} are satisfied. This is always possible because $M\mapsto \lambda_1(-d\Delta -(\alpha M - \mu))$ is a decreasing function that continuously goes from $\lambda_1(-d\Delta +\mu)>0$ to $-\infty$ as $M$ goes from $0$ to $+\infty$.

Let $(I^\star,R^\star)$ be the stationary solution of \eqref{stat} for such $\gamma^\star$, provided by Theorem~\ref{th existence}. Define
$$
S^\star(x) = M - I^\star (x) - R^\star (x).
$$
Then, we have
$$
-d\Delta I^\star = \alpha I^\star S^\star - \mu I^\star, \quad -d\Delta R^\star = -\lambda R^\star +\mu I^\star,
$$
that is, the two last equations of \eqref{PDE1} are verified, and in addition we have
$$
-d\Delta S^\star = - \alpha S^\star I^\star +\lambda R^\star.
$$
Now, it remains to prove that $S^\star$, which is in $ C^\delta_{per}$ by definition, is strictly positive. Because $I^\star,R^\star >0$, the elliptic maximum principle applied to the last equation implies that $S^\star >0$ also. Hence, the triple
$$
(S^\star,I^\star,R^\star)
$$
is indeed a positive periodic stationary solution to \eqref{PDE1}. 
\end{proof}

We now turn to the following technical result that will be used in the next section - we state it here as it concerns only the operators $T,A,Z$. We denote $\|\cdot\|_{L^2} = \|\cdot\|_{L^2([0,1]^N)}$ the $L^2$ norm on the periodicity cell $[0,1]^N$.
\begin{prop}\label{prop contractivity}
    Let $I_1,I_2 \in \mc I$. Assume that $\min\{ \frac{\gamma^\star}{\alpha} \} > \max\{\frac{\mu}{\lambda}\}\max\{\frac{\gamma^\star}{\alpha} \}$. Then
    $$
    \|T\circ A\circ Z(I_1) - T\circ A\circ Z(I_2) \|_{L^2} \leq  C_{TAZ}\|I_1 - I_2\|_{L^2},
    $$
    where $C_{TAZ}=\left(\frac{\max\{\alpha\}\max\{ \frac{\gamma^\star}{\alpha}\}\max\{\mu\}}{\min\{\alpha\}\min\{\lambda\}\left( \min\{ \frac{\gamma^\star}{\alpha} \} - \max\{\frac{\mu} {\lambda}\}\max\{\frac{\gamma^\star}{\alpha} \}\} \right)}\right)$.
\end{prop}
Before proving this result, we mention that the operator $T\circ A \circ Z$ is Lipschitz over $\mc I$ even without the condition $\min\{ \frac{\gamma^\star}{\alpha} \} > \max\{\frac{\mu}{\lambda}\}\max\{\frac{\gamma^\star}{\alpha} \}$, the goal of Proposition \ref{prop contractivity} is to have an explicit bound on $C_{TAZ}$ that we will use later.
\begin{proof}
The proof is split in three steps. 
\paragraph{Step 1: Lipschitz estimate on $T$.} Let us take $V_1,V_2\in A\circ Z(\mc I)$ and let us prove that
    $$
    \|T(V_1)-T(V_2)\|_{L^2} \leq \frac{\max\{\alpha\} \max\{\ol I\}}{\min\{\alpha\} \min\{\ul I\}} \|V_1 - V_2\|_{L^2}.
    $$
    We define $I_1 = T(V_1)$ and $I_2 = T(V_2)$. Then, we have
    $$
    -d\Delta(I_1 - I_2) = \alpha V_1 (I_1 - I_2) + \alpha (V_1- V_2) I_2 - \alpha (I_1+I_2)(I_1 - I_2),
    $$
    hence
    \begin{equation}\label{eq a int}
    -d\Delta(I_1 - I_2) - \alpha (V_1 - I_1)(I_1-I_2) +\alpha I_2 (I_1 - I_2) = \alpha (V_1- V_2) I_2.
    \end{equation}
    Multiplying this equation by $(I_1-I_2)$ and integrating over $[0,1]^N$, we get
    \begin{multline*}
    \int_{[0,1]^N}\left(d\vert \nabla(I_1 - I_2)\vert^2 - \alpha (V_1 - I_1)(I_1-I_2)^2\right) +\int_{[0,1]^N}\alpha I_2 (I_1 - I_2)^2 \\ = \int_{[0,1]^N}\alpha (V_1- V_2)(I_1-I_ 2) I_2.
    \end{multline*}
    Now, by definition we have $-d\Delta I_1 = \alpha (V_1 - I_1) I_1$ with $I_1>0$ and periodic. This means that $I_1$ is an eigenfunction for operator $-d\Delta -\alpha(V_1 - I_1)$, associated with the eigenvalue $0$. However, the Krein-Rutman theorem tells us that the only eigenvalue associated with a positive eigenfunction is the principal eigenvalue. Owing to the Rayleigh-Ritz formula, this implies that
    $$
    0\leq \int_{[0,1]^N}\left(d\vert \nabla(I_1 - I_2)\vert^2 - \alpha (V_1 - I_1)(I_1-I_2)^2\right).
    $$
    Therefore, we get
    $$
    \int_{[0,1]^N}  \alpha I_2 (I_1 - I_2)^2 \leq \int_{[0,1]^N} \alpha I_2 (V_1-V_2)(I_1 - I_2),
    $$
    and then
    $$
    \min\{\alpha I_2\} \|I_1 - I_2\|_{L^2} \leq \max\{\alpha I_2\} \|V_1 - V_2\|_{L^2}.
    $$
Now, because $V_1,V_2 \in A\circ Z(\mc I)$, we have $T(V_1),T(V_2)\in \mc I$ ($\mc I$ is stable under $T\circ A\circ Z$ owing to Lemma \ref{lem stable}), hence 
$$
\min \ul I \leq I_1,I_2 \leq \max \ol I,
$$
therefore
$$
\|T(V_1) - T(V_2)\|_{L^2} \leq \frac{\max\{\alpha\} \max\{\ol I\}}{\min\{\alpha\} \min\{\ul I\}}\|V_1 - V_2\|_{L^2}.
$$

\medskip
    \paragraph{Step 2: Estimate for $Z$.}
Let $I_1,I_2 \in \mc I$ and denote $R_1 = Z(I_1), R_2 = Z(I_2)$. Clearly, we have
$$
-d\Delta(R_1 - R_2) + \lambda (R_1 - R_2) = \mu(I_1-I_2).
$$
Multiplying by $(R_1-R_2)$ and integrating this equation, we get
$$
\|Z(I_1) - Z(I_2)\|_{L^2}\leq \frac{\max\{\mu\}}{\min\{\lambda\}} \|I_1 - I_2\|_{L^2.}
$$
Combining the estimates from the first and second steps we have that, for $I_1,I_2 \in \mc I$, we have
    $$
    \|T\circ A\circ Z(I_1) - T\circ A\circ Z(I_2) \|_{L^2} \leq  \left(\frac{\max\{\alpha\}\max\{\ol I\}\max\{\mu\}}{\min\{\alpha\}\min\{\ul I\}\min\{\lambda\}}\right)\|I_1 - I_2\|_{L^2}.
    $$

\medskip  
    \paragraph{Step 3: Conclusion.}
    To conclude, we need to estimate $ \max\{\ol I\}$ and $\min\{\ul I\} $ using the parameters. Recall that $\ol I$ solves $-d\Delta \ol I = \gamma^\star \ol I - \alpha \ol I^2$. Denoting $x_0$ a point where $\ol I$ reaches its maximum, we have $\gamma^\star(x_0) \ol I(x_0) - \alpha(x_0) \ol I^2(x_0) \geq 0$, so that 
    $$
    \max\{\ol I\} \leq \max\{\frac{\gamma^\star}{\alpha}\}. 
    $$
    Similarly, we find that
    $$
    \max\{\ol R\} \leq \max\{\frac{\mu}{\lambda}\}\max\{\ol I\}\leq \max\{\frac{\mu}{\lambda}\}\max\{\frac{\gamma^\star}{\alpha}\}
    $$
    and 
    $$
    \min\{\ul I\} \geq \min\{\frac{\gamma^\star}{\alpha} - \ol R\} \geq \min\{ \frac{\gamma^\star}{\alpha} \} - \max\{\frac{\mu}{\lambda}\}\max\{\frac{\gamma^\star}{\alpha}\}.
    $$
    Therefore, using these two estimates together with the assumption $\min\{ \frac{\gamma^\star}{\alpha} \} > \max\{\frac{\mu}{\lambda}\}\max\{\frac{\gamma^\star}{\alpha} \}$, we obtain
    \begin{multline*}
    \|T\circ A\circ Z(I_1) - T\circ A\circ Z(I_2) \|_{L^2} \leq  \left(\frac{\max\{\alpha\}\max\{\ol I\}\max\{\mu\}}{\min\{\alpha\}\min\{\ul I\}\min\{\lambda\}}\right)\|I_1 - I_2\|_{L^2}\\
     \leq  \left(\frac{\max\{\alpha\}\max\{\frac{\gamma^\star}{\alpha}\}\max\{\mu\}}{\min\{\alpha\}\left(\min\{ \frac{\gamma^\star}{\alpha} \} - \max\{\frac{\mu}{\lambda}\}\max\{\frac{\gamma^\star}{\alpha} \}\right)\min\{\lambda\}}\right)\|I_1 - I_2\|_{L^2},
    \end{multline*}
which concludes the proof.
\end{proof}

In the next lemma, we show that if $\lambda > \Lambda_0$ is verified (which is required in Theorem \ref{main th}), the operator $T\circ A \circ Z$ is a contraction. This will be useful in the next section.

We also prove that under this same condition, the assumption \eqref{assumption2} (required by Theorem \ref{th existence}) automatically holds true (hence, under Theorem \ref{main th}'s hypotheses, the system \eqref{stat} with $\gamma^*$ given by \eqref{def gamma} admits at least one positive periodic solution).


\begin{lemma}\label{lem Lambda}
Let $\Lambda_0$ be defined by \eqref{est lambda}. Assume that \eqref{assumption1} is verified and that $\lambda > \Lambda_0$.
    \begin{itemize}

        \item Assumption \eqref{assumption2} is verified. 

         \item The operator $T\circ A\circ Z$ is a contraction. 
    \end{itemize}
\end{lemma}
\begin{proof}
By definition of $\lambda_1(-d\Delta - (\gamma^\star - \alpha \ol R))$, there is $\phi \in C^\delta_{per}$, $\phi>0$ such that
$$
-d\Delta\phi  - (\gamma^\star - \alpha \ol R)\phi = \lambda_1(-d\Delta - (\gamma^\star - \alpha \ol R))\phi.
$$
Dividing this by $\phi$ and integrating by part, we find that
$$
-d \int_{[0,1]^N}\frac{\vert\nabla\phi\vert^2}{\phi^2} - \int_{[0,1]^N}(\gamma^\star - \alpha \ol R) = \lambda_1(-d\Delta - (\gamma^\star - \alpha \ol R)),
$$
therefore
\begin{eqnarray*}
    \lambda_1(-d\Delta - (\gamma^\star - \alpha \ol R)) &\leq& -\min\{\gamma^\star - \alpha \ol R\} \\ &\leq& -\min\{\alpha\}\min\left\{\frac{\gamma^\star}{\alpha} - \ol R\right\} \\
    &\leq&-\min\{\alpha\}\left(\min\{ \frac{\gamma^\star}{\alpha} \} - \max\{\frac{\mu}{\lambda}\}\max\{\frac{\gamma^\star}{\alpha}\} \right) \\ &\leq& -\min\{\alpha\}\left(\min\{ \frac{\gamma^\star}{\alpha} \} - \left\{\frac{\max \mu}{\min \lambda}\right\}\max\{\frac{\gamma^\star}{\alpha}\} \right)
\end{eqnarray*}
where the used the estimate on $\ol R$ obtained in the proof of Proposition \ref{prop contractivity}, step 3.

Therefore, if $\lambda > \Lambda_0$, then we have in particular $\lambda >\frac{\max\{\mu\}\max\{\frac{\gamma^\star}{\alpha}\}}{\min\{\frac{\gamma^\star}{\alpha}\}}$ and then $\lambda_1(-d\Delta - (\gamma^\star - \alpha \ol R)) < 0$, that is, \eqref{assumption2} holds true.

The fact that the Lipschitz constant $C_{TAZ}$ in \eqref{prop contractivity} is strictly lesser than $1$ when $\lambda > \Lambda_0$ follows from an easy computation.
\end{proof}

\section{Long-time behavior of the evolution system.}\label{sec evol}

In this section, we consider the system \eqref{PDE1} with a given initial datum $(S_0,I_0,0)$ satisfying the hypotheses of Theorem \ref{main th}. Take $\gamma^\star$ to be the function defined in \eqref{def gamma}.

We start with proving the first point of Theorem \ref{main th}. From now on, we assume that \eqref{assumption1} holds true and that $\lambda > \Lambda_0$, where $\Lambda_0$ is given by \eqref{est lambda}. Lemma \ref{lem Lambda} implies then that \eqref{assumption2} holds true. Therefore, we can apply Theorem \ref{th existence} to find
$$
(I^\star,R^\star)\in C^\delta_{per}, 
$$
positive periodic stationary solution to \eqref{stat} which is a fixed point of $T\circ A \circ Z$.

We now prove the convergence of the solutions of \eqref{PDE2} toward $(I^\star,R^\star)$, together with estimates on the speed of propagation.\\

Again, we denote $\ol I = T(\frac{\gamma^\star}{\alpha})$ and $ \ul I = T\circ A\circ Z(\ol I)$ as in the previous section. These functions are positive because of the assumptions \eqref{assumption1}, \eqref{assumption2}. We also recall that we denote $w^\star, w_\star$ the speeds of propagation given by \eqref{def w haut}, \eqref{def w bas}.

We start with some preliminary definitions and reminders from regularity theory for parabolic equations.

\subsection{Definitions}

We define the space $C^\delta_{loc.unif}$ of functions which are locally uniformly Hölder continuous. To do so, we denote, for $R>0$ and $(t,x) \in \R^{N+1}$, the parabolic cylinder $Q_R(t,x) := (t,x)+(-R^2,0]\times B_R(0)$. For a function $u(t,x)$ we denote $\|u\|_{\delta,Q_R(t,x)} $ the parabolic Hölder norm, that is,
$$
\|u\|_{\delta,Q_R(t,x)} =\|u\|_{L^\infty(Q_R(t,x))}+\sup_{(t_1,x_1)\neq (t_2,x_2) \in Q_R(t,x)}\frac{|u(t_1,x_1) - u(t_2,x_2)|}{(|t_1 - t_2|^{\frac{1}{2}}+|x_1 - x_2|)^\delta}.
$$
We also denote the $C^{2,\delta}$ norm by $\|u\|_{2+\delta,Q} = \|\partial_t u\|_{\delta,Q} +\|D^2 u\|_{\delta,Q} + \|Du\|_{L^\infty(Q)}+\|u\|_{L^\infty(Q)}$.
Then, we introduce
\begin{multline*}
C^{\delta}_{loc.unif} :=\{ u \in C^{1,2}(\R^+\times\R^N)\ : \ \forall R>0, t>R^2+ 1, \ x\in \R^N, \\ \exists C_R>0\ \text{independent of $(t,x)$} \ :\ \|u\|_{\delta,Q_R(t,x)} \leq C_R\}.
\end{multline*}
 We emphasize that the constant $C_R$ in the definition of $C^\delta_{loc.unif}$ is independent of the center of the parabolic cylinder $(t,x)$ (but may depend on its size).

\vspace{0.2cm}
Now let $\mc S(p,\ul w, \ol w)$ be the space of functions which converge to $p\in C^\delta_{per}$ with speed at least $\ul w$ and at most $\ol w$,
\begin{multline*}
\mc S(p,\ul w, \ol w) :=\Big\{ u\geq 0,\ u\in L^\infty\cap C^\delta_{loc.unif} \ : \  \forall \e>0, \sup_{\substack{x = r e \\ r\in [0,(\ul w(e)-\e)t], e\in \S}} \vert u - p\vert \cv 0 , \\ \ \sup_{\substack{x = r e \\ r\geq (\ol w(e)+\e)t, e\in \S}}  \vert u \vert \cv 0\Big\}.
\end{multline*}
 The space $C^\delta_{loc.unif}$ will be useful as we will often work with translated functions of the form $u_n(t,x) = u(t+t_n,x+t_n)$. If $u$ is locally Hölder continuous, so is $u_n$, but the Hölder constant could depend on the translation point $(t_n,x_n)$. Having $u\in C^{\delta}_{loc.unif}$ ensures that the sequence $(u_n)_{n\in \N}$ is locally Hölder continuous with a Hölder constant that does not depend on $n$.

 We recall the following standard regularity result
\begin{prop}\label{prop reg}
Let $R>0,t_0,x_0\in\R^+\times \R^N$ be fixed and let $Q = Q_R(t_0,x_0)$ and $\t Q = Q_{R/2}(t_0,x_0)$ be two parabolic cylinder. Let $c,f \in C^\delta(Q)$ and let $u(t,x) \in C^{1,2}(Q)$ be such that
    $$
    \partial_t u -d\Delta u - cu= f(t,x),\quad (t,x) \in Q.
    $$
    Then, there is $C>0$ (depending only on $\|c\|_{\delta,Q},N,d,\delta$ and $R$), such that
        $$
        \|u\|_{2+\delta,\t Q}\leq C(\|f\|_{\delta,Q}+\|u\|_{L^\infty(Q)}).
        $$
\end{prop}
We refer to \cite{Lieb}, Chapter 4, for a proof of this fact. Observe that $C^\delta_{per} \subset C^\delta_{loc.unif}$.

\subsection{Preliminary results.}

This section contains two lemmas that characterize how the solutions of the first and second equations of \eqref{PDE2} behave when freezing the functions $I$ and $R$ respectively. This will be the core of our bootstrap argument.

\begin{lemma}\label{lem lin}
Let $p \in C^\delta_{per}$ be strictly positive. Let $0<\ul w \leq \ol w $ be continuous positive functions from the sphere $\S$ to $\R^+_\star$. If $u\in \mc S(p,\ul w,\ol w)$, and if $v$ satisfies
    $$
    \partial_t v = d\Delta v - \lambda v +\mu u, \quad t>0,\ x\in \R^N,
    $$
    with initial datum $v_0$ continuous, non-negative and compactly supported, then
    $$
    v \in S(Z(p),\ul w,\ol w).
    $$
\end{lemma}

\begin{proof}
First, observe that $v$ is uniformly bounded. Indeed, the constant function everywhere equal to 
$$
\max\left\{\|v_0\|_{L^\infty(\R^N)}, \left\|\frac{\mu u}{\lambda}\right\|_{L^\infty(\R^+\times \R^N)}\right\}
$$
is supersolution of the equation satisfied by $v$ and is larger than $v$ at initial time.

Moreover, $\mu \in C^\delta_{per}$ and $u\in C^\delta_{loc.unif}\cap L^\infty$, then $\mu u \in C^\delta_{loc.unif}\cap L^\infty$. The regularity result Proposition \ref{prop reg} implies that we have $v\in C^{\delta}_{loc.unif}$.

    Now, it remains to show that, for any $\e>0$,
    \begin{equation}\label{lem int}
    \sup_{\substack{x = re \\ r\in [0,(\ul w(e)-\e)t], e\in \S} } \vert v - Z(p)\vert \to 0,
     \end{equation}
    and
    \begin{equation}\label{lem ext}
    \sup_{\substack{x = re \\ r\geq (\ol w(e)+\e)t, e\in \S} } \vert v \vert \to 0.
   \end{equation}

\paragraph{Step 1. Proof of \eqref{lem int}: reduction to a problem on entire solutions.} Let $\e>0$ be chosen and take two sequences $(t_n)_{n\in \N},(x_n)_{n\in \N}$ such that $t_n \to+\infty$, $x_n \in \{x = re,  \ r \in [0,(\ul w(e)-\e)t_n], \ e\in \S\}$, and
$$
 \vert v(t_n,x_n) - Z(p)(x_n)\vert  = \sup_{\substack{x = re \\ r\in [0, (\ul w(e)-\e)t_n]} } \vert v - Z(p)\vert.
$$
We define two sequences $(k_n)_{n\in \N} \in \Z^N$ and $(y_n)_{n\in \N} \in [0,1)^\N$ so that $x_n = k_n + y_n$. Then, define the translations $v_n(t,x):= v(t +t_n, x +k_n)$ and $u_n(t,x):= u(t+t_n,x + k_n)$. The function $u_n$ satisfies (using that $\lambda,\mu$ are periodic)
    $$
    \partial_t v_n = d\Delta v_n - \lambda v_n +\mu u_n,\quad t>-t_n,\ x\in \R.
    $$
    Since $u\in \mc S(p,\ul w, \ol w)$, $u_n(t,x) \underset{n\to+\infty}{\longrightarrow} p$, locally uniformly.
    In addition, Proposition \ref{prop reg} implies that $v_n$ is $C^{2+\delta}$ on every compact with a constant independent of $n$, since $ - \lambda v_n +\mu u_n$ is bounded uniformly in $t,x$ (independently of $n$) and belongs to $C^\delta_{loc. unif}$. Hence, up to extraction, the Arzelà-Ascoli theorem guarantees that $v_n$, $\partial_t v_n$, $\Delta v_n$ converge uniformly in $(t,x)$ on every compact of $\R\times \R^N$ to $v_\infty$, which solves 
    $$
    \partial_t v_\infty - d\Delta v_\infty + \lambda v_\infty =\mu p, \text{ for all } t \in \mathbb{R}, x\in \R^N.
    $$
    Moreover, up to another extraction, denoting $y$ a limit point of the bounded sequence $(y_n)_{n\in\N}$, we have
$$
 \vert v_\infty(0,y) - Z(p)(y)\vert  = \limsup_{t_n\to+\infty}\sup_{\substack{x = re \\ r\in [0, (\ul w(e)-\e)t_n]} } \vert v - Z(p)\vert.
$$
Let us now prove that $v_\infty \equiv Z(p)$, which would give the desired result.

\paragraph{Step 2. Proof of \eqref{lem int}.} Clearly $v_\infty$ is bounded uniformly in $(t,x)$ and is non-negative (because all the $v_n$ are). Let $T>0$ be fixed and define $\ul v^T(t,x)$ as the solution of 
\begin{equation}\label{vt}
    \partial_t \ul v^T - d\Delta \ul v^T + \lambda \ul v^T =\mu p, \text{ for all } t >-T, x\in \R,
\end{equation}
    with initial datum $\ul v^T(-T,\cdot)\equiv 0$. By comparison, we have
    $$
    \ul v^T (t,\cdot)\leq v_\infty(t,\cdot),\quad t>-T.
    $$
    The function $\ul v^T$ converges to $Z(p)$ as $t$ goes to $+\infty$. Indeed, the function everywhere constant equal to zero is subsolution of the stationary equation associated to \eqref{vt}, and it is classical (c.f. the lecture notes \cite{RR}) that $\ul v^T$ is then time increasing and converges (pointwise, and then $C^2$ uniformly owing to regularity theory) toward a stationary solution of $\eqref{vt}$. The unique stationary solution of \eqref{vt} is precisely $Z(p)$.
    
    Now, observing that, for $(t,x)$ fixed, $\ul v^T(t,x)$ also converges toward the same limit as $T$ goes to $+\infty$ (this comes from the fact that \eqref{vt} is autonomous), we find that
    $$
    Z(p) \leq v_\infty.
    $$
    The proof of the reverse inequality is similar. We define $\ol v^T$ to be solution of the same equation but with initial datum $\ol v^T(-T,x) = M $, where $M>0$ is a positive constant, large enough to be supersolution of the stationary equation (and to be above $v_\infty$). Using the same argument we end up finding $Z(p)\geq v_\infty$. Therefore, $ \limsup_{t\to+\infty}\sup_{\substack{x = re \\ r\in [0, (\ul w(e)-\e)t]} } \vert v - Z(p)\vert=0$.

    \paragraph{Step 3. Proof of \eqref{lem ext}.} The argument is similar: we take a fixed $\e>0$ and two sequences $(t_n)_{n\in \N},(x_n)_{n\in \N}$ such that $t_n \to+\infty$, $x_n \in \{x = re,  \ r \geq (\ol w(e)+\e)t_n, \ e\in \S\}$. We define the translations $u_n,v_n$ as above, we have that $v_n\to v_\infty$ solution of
$$
    \partial_t v_\infty - d\Delta v_\infty + \lambda v_\infty =0, \text{ for all } t \in \mathbb{R}, x\in \R^N.
    $$
    Following the same lines we end up finding that $v_\infty=0$ and this proves the result.
\end{proof}

In the next lemma, we use the speeds $w^\star, w_\star$, defined in \eqref{def w haut}, \eqref{def w bas}.

\begin{lemma}\label{lem KPP}
Let $p\in C^\delta_{per}$ be strictly positive and such that $p\leq \ol R$. Let $v\in \mc S(p, w_\star, w^\star)$ and $u$ be solution to
    \begin{equation}\label{eq u}
    \partial_t u - d\Delta u = (\gamma - \alpha v) u -\alpha u^2, \quad t>0,\ x\in \R,
    \end{equation}
    with initial datum $u_0$ continuous, compactly supported, non-negative and non-zero. Then,
    $$
    u \in \mc S\left(T\circ A(p),  w_\star, w^\star \right).
    $$
\end{lemma}

\begin{proof} We split the proof into three steps.

\paragraph{Step 1. $u$ spreads at most with speed $w^\star$.}

Observe that, since $\gamma$ converges to $\gamma^\star$ uniformly when $t \to +\infty$, for all $\eta>0$ there exists $T>0$ such that
\begin{equation}
    \label{eq encadrement gamma}
   \gamma(t,\cdot) \leq \gamma^\star+ \eta=: \gamma^{\star}_{\eta},\quad \forall t>T.
\end{equation}
Then, since $v\geq 0$, we have
$$
\partial_t u - d\Delta u \leq \gamma^{\star}_{\eta} u - \alpha u^2, \qquad \text{for all } t>T.
$$
Hence the solution $\tilde{u}$ of 
$$
 \partial_t \t u - d\Delta \t u =\gamma^{\star}_{\eta} \t u - \alpha \t u^2,\quad t>T
$$
with
$$
\t u(T,\cdot) = u(T,x).
$$
is a supersolution of \eqref{eq u} for $t>T$, which implies that $u \leq \t u $ for all $t>T$.

However, Theorem \ref{th FG}\footnote{Observe that $u(T,\cdot)$ is not compactly supported. This is not a problem: Theorem \ref{th FG} still holds when the initial data have a Gaussian decay (this is the case of $u(T,\cdot)$).} tells us that $\t u$ converges toward $\ol I_\eta = T\left(\frac{\gamma^*_\eta}{\alpha}\right)$ solution of 
$$
-d\Delta \ol I_\eta=  \gamma^{\star}_{\eta} \ol I_\eta - \alpha (\ol I_\eta)^2,
$$
with speed of propagation $w^{\star}_{\eta}$ given by
$$
w^{\star}_{\eta} (e) := \inf_{\lambda \in \mathbb{R}^N} \frac{-k(\lambda) + \eta}{\lambda\cdot e}
$$
where $k(\lambda)$ is the principal eigenvalue of the operator $L_\lambda$, \eqref{eq operator L_lambda}. 
Hence, for all $\e>0$
$$
\sup_{\substack{x = re \\ r\geq (w^{\star}_\eta(e)+\e)t}}\vert u \vert \underset{t\to+\infty}{\longrightarrow} 0.
$$
and, since this is true for all $\eta >0$, it still holds by replacing $w^{\star}_{\eta}$ by $w^{\star}$.\\

\paragraph{Step 2. $u$ converges to $T\circ A (p)$ with speed at least $ w_\star$.}

Let $\e>0$. As in the proof of the previous lemma, we take two sequences $(x_n)_{n\in \N}$ and $(t_n)_{n\in \N}$ such that $t_n\to+\infty$, $x_n \in \{x = re \ : \  r\leq (w_\star(e)-\e)t_n, \ e\in \S\}$ and
$$
\vert u(t_n,x_n)-T\circ A(p)\vert = \sup_{\substack{x = re \\ r\leq (w_\star(e)-\e)t_n}} \vert u(t_n,x) - T\circ A(p)\vert.
$$
As before, we also define $(k_n)_{n\in \N} \in (\Z^N)^\N$ and $(y_n)_{n\in \N}\in ([0,1)^N)^\N$ such that $x_n = k_n+y_n$ and the translations
$$
u_{n}(t,x) = u(t+t_n,x+k_n),\quad v_n(t,x) = v(t+t_n,x+k_n).
$$
Using the same arguments as in Lemma \ref{lem lin}, because $v\in \mc S(p, w_\star, w^\star)$, we have that $u_n$ converges to a solution $u_\infty$ of
\begin{equation}\label{eq u inf}
\partial_t u_\infty = d\Delta u_\infty + (\gamma^\star-\alpha p) u_\infty-\alpha u_\infty^2, \quad t\in \R, x\in \R^N,
\end{equation}
due to \eqref{def gamma}.
As before, letting $y$ be a limit point of the sequence $(y_n)_{n\in \N}$, we have
$$
\vert u_\infty(0,y)-T\circ A (p)(y)\vert =\limsup_{n\to+\infty} \sup_{\substack{x = re \\ r \leq (w_\star(e) - \e)t_n}}\vert u(t_n,x)-T\circ A(p) \vert.
$$
To conclude the proof, we now show that $u_\infty \equiv T\circ A(p)$. For simplicity, we denote $\t p =T\circ A(p)$.

Observe that the function $M \t p$ is supersolution of \eqref{eq u inf} if $M>1$. We take $M>1$ large enough so that $M \t p \geq u_\infty$ (which is possible because $u_\infty$ is bounded). Let $\t u$ be the solution of \eqref{eq u inf} with initial datum $\t u(-T,\cdot) = M \t p$, for some $T>0$. Since the initial datum is a stationary supersolution, $\t u$ is time non-increasing, hence converges toward a bounded stationary solution. Arguing as in the previous lemma, we find
$$
u_\infty \leq \t p = R\circ A(p).
$$
Let us now prove the reverse inequality. To do so, we can argue as above: observe that the function $M \t p$ is stationary subsolution of \eqref{eq u inf} when $M\in [0,1)$.

Now, for $T>0$, define $\t u$ to be the solution of \eqref{eq u inf} with initial datum $\t u(-T,\cdot) = M \t p$.

Assume for now that there is $\rho >0$ such that
\begin{equation}\label{u inf}
    u_\infty(t,x) \geq \rho \quad \forall t\in \R,\ x\in \R^N.
\end{equation}
Then, up taking $M$ small enough so that $M\t p \leq \rho$, we have that $\t u \leq u_\infty$ for all $t>-T$, and arguing as above we find that that $\t p \leq u_\infty$, finishing the proof. The key point is to have \eqref{u inf}.

\paragraph{Step 3. Proof of \eqref{u inf}.} The proof of \eqref{u inf} relies not directly on the Freidlin-Gartner Theorem \ref{th FG} but on a technical point of its proof from \cite{BHN}.

Indeed, recall that the function $u_\infty$ is the limit of $u(t+t_n,x+k_n)$. The function $u$ itself solves \eqref{eq u}.

Let $\delta,\eta \in (0,\frac{\e}{2})$. Since $\gamma(t,x)$ converges uniformly to $\gamma^\star$ as $t$ goes to $+\infty$ and $v \in \mc S(p,w_\star,w^\star)$, we have that, for $T>0$ large enough,
\begin{equation}\label{super sol tech}
\partial_t u \geq d\Delta u +(\gamma^\star - \alpha p - \eta)u - \alpha u^2,\quad \text{for}\ t>T,\ x\in \mc W_t^\delta,
\end{equation}
where we denote $\mc W_t^\delta :=\{x = re \ : \ r \leq (w_\star(e) - \delta)t,\ e\in \S\}$.

Now, let $w_\eta$ be the speed of propagation associated to the equation \eqref{super sol tech} given by the Freidlin-Gartner theorem (when the equation is set on the whole space $x\in \R^N$, that is). Up to taking $\eta$ small enough, we can guarantee that $w_\eta \geq w_\star -\frac{\e}{2}$.

It is then a consequence of \cite{BHN}\footnote{More precisely, this is a consequence of the proof of Theorem 1.13, section 4.1 in \cite{BHN}. The authors prove the existence of subsolutions of \eqref{super sol tech} that are compactly supported and move in each directions $e$ with speed as close as we want to $w_\eta$. This is actually the key point of the proof of the Freidlin-Gartner theorem in \cite{BHN}.} that there are $\rho,T>0$ such that
$$
u(t,x) \geq \rho,\quad \text{for}\ t>T, \ x\in \mc W_t^{\frac{\e}{2}}.
$$
Then, by definition of the sequences $t_n, k_n$ chosen in step $2$, for any $(t,x)\in \R\times \R^N$, there is $n_0$ such that $n\geq n_0$ implies that $x+k_n \in \mc W_{t+t_n}^{\frac{\e}{2}}$. This means that $u_n(t,x) \geq \rho$ for such $n$, and finally $u_\infty(t,x)\geq \rho$. Because this is true for all $(t,x)$, \eqref{u inf} is verified.

\end{proof}

\subsection{Conclusion}

The core of the proof of Theorem \ref{main th} is the following.

\begin{prop}\label{prop conc}
Assume that the hypotheses of Theorem \ref{main th}, first point, are verified, i.e. \eqref{assumption1} and $\lambda >\Lambda_0$ hold true. Let $(I,R)$ be the solution of \eqref{PDE2} with initial datum $(I_0,0)$ where $I_0$ is non-negative, non-zero and compactly supported and let $p \in C^\delta_{per}$ be a strictly positive function such that $p \in \mathcal{I}$ i.e.
$$
\ul I  \leq p \leq \ol I.
$$
Assume there exists a function $u_0\in \mc S(p, w_\star,w^\star)$ such that    
    $$
    I\leq u_0.
    $$
    Then, there are $u_1,u_2,u_3,u_4$ such that
    $$
    u_3 \leq R \leq u_1,
    $$
    $$
    u_2 \leq I \leq u_4
    $$
    with
    \begin{equation*}
    \left\{
    \begin{array}{ll}
     &u_1 \in \mc S(Z(p), w_\star,w^\star),\\
     &u_2 \in \mc S(T\circ A\circ Z(p),w_\star,w^\star), \\
     &u_3 \in \mc S(Z\circ T\circ A\circ Z(p),w_\star,w^\star),\\
     &u_4 \in \mc S( (T\circ A\circ Z)^2(p),w_\star,w^\star).
    \end{array}
    \right.
    \end{equation*}
\end{prop}

\begin{proof}
   We have $I \leq u_0$ where
$$
u_0 \in \mc S(p, w_\star,w^\star).
$$
Therefore,
$$
\partial_t R \leq d\Delta R -\lambda R +\mu u_0,  \quad t>0, \ x\in \R^N.
$$
Hence, letting $u_1$ be the solution of 
$$
\partial_t u_1 = d\Delta u_1 -\lambda u_1 +\mu u_0,  \quad t>0, \ x\in \R^N,
$$ with initial datum $u_1(0,\cdot) \equiv 0$, we find thanks to the parabolic comparison principle that
$$
R \leq u_1,
$$
and, thanks to Lemma \ref{lem lin}, we have $u_1 \in \mc S(Z(p), w_\star,w^\star)$. This implies in turn that
$$
\partial_t I \geq d \Delta I +(\gamma -\alpha u_1)I - \alpha I^2,  \quad t>0, \ x\in \R^N.
$$
Therefore, owing to the parabolic comparison principle, we have $I \geq u_2$, where $u_2$ solves
$$
\partial_t u_2 = d\Delta u_2 +(\gamma - \alpha u_1)u_2 - \alpha u_2^2,  \quad t>0, \ x\in \R^N,
$$
with initial datum $u_2(0,\cdot) = I_0$.

Because $p\leq \ol I$, we have $Z(p) \leq \ol R$. We can then apply
Lemma \ref{lem KPP} to obtain $u_2 \in \mc S(T\circ A\circ Z(p), w_\star,w^\star)$.

Now, this implies that
$$
\partial_t R \geq d\Delta R - \lambda R +\mu u_2,  \quad t>0, \ x\in \R^N,
$$
and by the same arguments as above this implies that $R \geq u_3$, which solves $\partial_t u_3 = d\Delta u_3 - \lambda u_3 +\mu u_2$, and $u_3 \in \mc S\left(Z\circ T\circ A\circ Z(p),w_\star,w^\star\right)$.\\

This in turns implies that $\partial_t I \leq d\Delta I +(\gamma - \alpha u_3)I - \alpha I^2$ for $t>0, \ x\in \R^N$. Because $p\geq \ul I$, we have that $Z\circ T\circ A\circ Z(p) \leq Z\circ T\circ A\circ Z(\ul I) \leq Z(\ol I) = \ol R$, we can then apply Lemma \ref{lem KPP}. We find that $I \leq u_4$, where $u_4 \in \mc S\left(( T\circ A\circ Z)^2(p),w_\star,w^\star\right)$. This concludes the proof.
\end{proof}

We are now in position to obtain Theorem \ref{main th}. This is done by iterating Proposition \ref{prop conc}.

\begin{proof}[Proof of Theorem \ref{main th}]
    Let $I,R$ be the solution of \eqref{PDE2} arising from the initial datum $(S_0,I_0,0)$. First, observe that
    $$
    \partial_t I \leq d\Delta I +\gamma I -\alpha I^2,  \quad t>0, \ x\in \R^N.
    $$
    Let $u_0(t,x)$ the solution of
    $$
    \partial_t u_0 = d\Delta u_0 +\gamma u_0 - \alpha u_0^2,\quad t>0,\ x\in \R,
    $$
    with initial datum $u_0(0,\cdot) = I_0$. We have, thanks to Lemma \ref{lem KPP} applied with $v =0 \in \mc S(0,w_\star,w^\star)$ (and observing that $T\circ A(0)=T(\frac{\gamma^\star}{\alpha})=\ol I$), 
    $$
    u_0 \in\mc S(\ol I,w_\star,w^\star).
    $$
    Therefore, Proposition \ref{prop conc} implies that there are two functions $\ol u,\ul u$ such that
    $$
    \ul u \leq I \leq \ol u
    $$
    with
    \begin{equation*}
    \left\{
    \begin{array}{ll}
     &\ul u \in \mc S(T\circ A\circ Z(\ol I), w_\star,w^\star), \\
     &\ol u \in \mc S((T\circ A\circ Z)^2(\ol I),w_\star,w^\star).
    \end{array}
    \right.
    \end{equation*}
    Iterating the argument, for all $n \in \N$, there are functions $\ul u_n, \ol u_n$ such that
     $$
    \ul u_n \leq I \leq \ol u_n
    $$
    with
    \begin{equation*}
    \left\{
    \begin{array}{ll}
     &\ul u_n \in \mc S((T\circ A\circ Z(\ol I))^{2n+1},w_\star,w^\star), \\
     &\ol u_n \in \mc S((T\circ A\circ Z(\ol I))^{2n+2},w_\star,w^\star).
    \end{array}
    \right.
    \end{equation*}
Owing to Proposition \ref{prop contractivity} and Lemma \ref{lem Lambda}, the operator $T\circ A\circ Z$ is a contraction on $\mc I$. Therefore, $(T\circ A\circ Z(\ol I))^{2n+1}$ and $(T\circ A\circ Z(\ol I))^{2n+2}$ converge toward the fixed point $I^\star$ given by Theorem \ref{th existence} when $n$ goes to $+\infty$.\\

Observe that these convergences hold not only in $L^2$ but also in $L^\infty$ norm: indeed, the operator $T\circ A\circ Z$ is compact from $L^2$ into $C^\delta$ (this result was obtained in the proof of Theorem \ref{th existence} as an application of elliptic regularity theory). Therefore, for $\eta>0$, up to taking $n$ large enough, we have, for all $x\in \R^N$, 
$$
I^\star(x) - \eta\leq (T\circ A\circ Z(\ol I))^{2n+1}(x)\leq (T\circ A\circ Z(\ol I))^{2n+2}(x)\leq I^\star(x) +\eta.
$$
Now, for $\e>0$ fixed, because $\ul u_n \in \mc S((T\circ A\circ Z(\ol I))^{2n+1},w_\star,w^\star)$ and $\ol u_n \in \mc S((T\circ A\circ Z(\ol I))^{2n+2},w_\star,w^\star)$, we can find $T>0$ large enough so that, for $t>T$, $x\in \{x = re\ : \ r\in [0,(w_\star(e)-\e)t],\ e\in \S\}$, we have
$$
\ol u_n(t,x) \leq  (T\circ A\circ Z(\ol I))^{2n+2}(x)+\eta
$$
and 
$$
\ul u_n(t,x) \geq (T\circ A\circ Z(\ol I))^{2n+1}(x) - \eta.
$$
Combining all that precedes, we find that, for $\eta>0$, there is $T>0$ such that, for $t>T$,
$$
 \sup_{\substack{x = re \\ r\leq (w_\star(e)-\e)t_n}} \vert I(t,x) - I^\star(x)\vert \leq 2\eta.
$$
Therefore,
$$\sup_{\substack{x = re \\ r\leq (w_\star(e)-\e)t}} \vert I(t,x) - I^\star(x)\vert \underset{t\to+\infty}{\longrightarrow}0.
$$
Similarly, we prove that
$$\sup_{\substack{x = re \\ r\geq (w_\star(e)+\e)t}} \vert I(t,x)\vert \underset{t\to+\infty}{\longrightarrow}0,
$$
and the same holds true for $R$. This proves the result.
\end{proof}

To conclude the proof of Theorem \ref{main th}, it remains to consider the second point, that is, we show that when
$$
\lambda_1(-d \Delta - \gamma^\star) \geq 0,
$$
the disease does not spread.
\begin{prop}\label{prop non prop}
    Assume that \eqref{assumption1} is not verified. Then, 
    $$
    (S,I,R) \underset{t\to +\infty}{\longrightarrow} \left( \fint S_0,0,0 \right),
    $$
    and this convergence is uniform in $x \in \R^N$.
\end{prop}
\begin{proof}
First, we know that $(I,R)$ solves \eqref{PDE2}, hence
    $$
    \partial_t I \leq d\Delta I + \gamma(t,x) I - \alpha I^2,  \quad t>0, \ x\in \R^N.
    $$
Define $u$ to be the solution of
$$
    \partial_t u = d\Delta u + \gamma(t,x) u - \alpha u^2,  \quad t>0, \ x\in \R^N,
    $$
    with initial datum $u(0,\cdot) = I_0$. Hence, $I\leq u$.
    
    Now, let us show that $u(t,x)$ goes to zero uniformly with respect to $x$ as $t$ goes to infinity. For any sequence $(t_n)_{n\in\N} \in \R_+^\N$ and any sequence $(x_n)_{n\in\N} \in (\R^N)^\N$, we denote $u_n(t,x)=u(t+t_n,x+k_n)$, where $(k_n)_{n\in\N} \in (\Z^N)^{\N}$ is such that $x_n = k_n + z_n$ with $z_n \in [0,1]^N$. We obtain $$
    \partial_t u_n = d\Delta u_n + \gamma(t+t_n,x+k_n) u_n - \alpha u_n^2, 
     \text{ for } t>-t_n.$$ 
    
    Observe also that, since $\alpha,\mu$ are periodic and $N(t,x)$ converges toward $\fint S_0$ uniformly, we have that
    $$
    \gamma(t+t_n,x+k_n)=\alpha(x)N(t+t_n,x+k_n)-\mu(x)
    $$
    converges uniformly toward $\gamma^\star(x)$. Hence, the parabolic regularity theory implies that, up to extraction, $u_n(t,x)$ converges toward $u_\infty(t,x)$ as $n$ goes to $+\infty$, where $u_\infty$ solves
    $$
    \partial_t u_\infty = d\Delta u_\infty + \gamma^\star(x) u_\infty - \alpha u_\infty^2, \quad t \in \R, x\in \R^N.
    $$
   However, the only non-negative bounded solution of this equation is $0$, because $\lambda_1(-d\Delta - \gamma^\star) \geq 0$ (this follows by the same argument used in the proof of Lemma \ref{lem lin} and the Freidlin-Gartner Theorem \ref{th FG}). Hence, $u(t_n,x_n) \underset{n\to+\infty}{\longrightarrow} 0$, and therefore, $I(t,x)$ converges to zero uniformly in $x$ as $t$ goes to $+\infty$.\\

 Now, since $R$ satisfies 
 $$\partial_t R = d \Delta R -\lambda R + \mu I$$ 
 and $I$ goes uniformly to zero, we also have that $R$ goes to zero uniformly as $t$ goes to $+\infty$. This in turn implies that $S(t,x) = N(t,x) - I(t,x) - R(t,x)$ indeed converges to $\fint S_0$ uniformly, hence the result.
\end{proof}

We now consider the homogeneous case, and we prove Corollary \ref{cor hom}.
\begin{proof}[Proof of Corollary \ref{cor hom}]
The proof mainly consists in computing effectively the quantities that appear in Theorem \ref{main th}. Clearly, since $S_0 \in \mathbb{R}$, $\gamma^\star = \alpha S_0 - \mu \in \R$ and the principal eigenvalue of the operator
$$
\phi \mapsto -d\Delta \phi - (\alpha S_0 - \mu)\phi
$$
is simply $- (\alpha S_0 - \mu)$ (associated with the principal eigenfunction constant equal to $1$ for instance).

The second point of Corollary \ref{cor hom} directly comes from Theorem \ref{main th}.\\

For the first point, observe that the constant functions
    $$
(S^\star,I^\star,R^\star) = \left(\frac{\mu}{\alpha} , \frac{\alpha S_0 - \mu}{\alpha(1+\frac{\mu}{\lambda})},\frac{\mu}{\lambda}\frac{\alpha S_0 - \mu}{\alpha(1+\frac{\mu}{\lambda})} \right)
    $$
are indeed stationary solutions of \eqref{PDE1} and satisfy $S^\star + I^\star + R^\star = S_0$.

It is easy to check that 
$$\ol I = S_0 - \frac{\mu}{\alpha}, \quad \ol R = \frac{\mu}{\lambda}\ol I, \quad  w^\star = 2\sqrt{d(\alpha S_0 - \mu)} \text{ and }  w_\star = 2\sqrt{d(\alpha S_0 - \mu)\left(1 - \frac{\mu}{\lambda}\right)}.$$
In particular, we have
$$
\lambda_1\left(-d\Delta - (\gamma^\star - \alpha \ol R) \right)=-\gamma^\star\left(1 - \frac{\mu}{\lambda}\right).
$$
The only thing one has to check is that assuming $\lambda > \mu$ is sufficient. Clearly, in the homogeneous case, the quantity $\Lambda_0$ from \eqref{est lambda} does not boil down to $\mu$ (but to $2\mu$, which is larger). However, coming back to the proof of Proposition \ref{prop contractivity}, one can observe that the first step (the Lipschitz estimate on $T$) is much more simpler than the one we obtained: the operator $T$ is indeed the identity (because we are in the case where $\gamma^\star =\alpha S_0 - \mu>0)$. Hence, we should have
$$
C_{TAZ} = \frac{\mu}{\lambda}.
$$
Therefore, as soon as $\lambda > \mu$, then \eqref{assumption2} is verified and $C_{TAZ}<1$, which means that Lemma \ref{lem Lambda} holds true, and the rest of the proof is unchanged.
\end{proof}

\paragraph{Acknowledgements:} This study contributes to the IdEx Université de Paris ANR-18-IDEX-0001. M.L. would like to acknowledge the support from the project GOTA ANR-23-CE46-0001-01 (2023-2027). The research leading to these results has received funding from the ANR project “ReaCh” (ANR-23-CE40-0023-01).

\bibliographystyle{plain} 
\bibliography{biblio}

\end{document}